\documentclass[11pt]{article}
\usepackage[left=1in,right=1in,top=1in,bottom=1in]{geometry}

\usepackage{subfigure}
\usepackage{graphicx}
\usepackage{amsfonts}
\usepackage{url}
\usepackage{amsmath,amssymb}
\usepackage{enumerate}
\usepackage{color}
\usepackage{amsthm,amsmath}
\usepackage{bbm}
\usepackage[version=4]{mhchem}
\usepackage{xspace}
\usepackage{multicol}

\usepackage{fancybox}

\usepackage{tikz}
\usetikzlibrary{arrows,shapes,positioning}
\usepackage{tikz-cd}

\usepackage[colorlinks]{hyperref}

\DeclareMathOperator{\Span}{span}
\DeclareMathOperator{\rank}{rank}
\DeclareMathOperator{\im}{im} 

\DeclareMathOperator{\cone}{cone}
\DeclareMathOperator{\relint}{relint}

\newcommand{\RR}{\mathbb{R}}
\newcommand{\dd}[2]{\frac{\text{d} #1}{\text{d} #2}}

\newcommand{\eq}{X}
\newcommand{\vd}{$V^*$-directed\xspace}
\newcommand{\fp}{\sigma} 
\newcommand{\fpp}{\tau}  
\newcommand{\tc}{\kappa} 

\newcommand{\tr}{\intercal} 
\newcommand{\rd}{\hat} 

\newcommand{\LRARR}[4]{{\mbox{ \raise 0.4 mm \hbox{$#1$}}} \;
  \mathop{\stackrel{\displaystyle\longrightarrow}\longleftarrow}^{#3}_{#4}
  \; {\mbox{\raise 0.4 mm\hbox{$#2$}}}}

\newcommand{\ra}[1]{\xrightarrow{\hspace*{#1cm}}}
\newcommand{\la}[1]{\xleftarrow{\hspace*{#1cm}}}
\newcommand{\LR}[5]{{\mbox{ \raise 1.3 mm \hbox{$#1$}}} \;
  \mathop{\stackrel{\displaystyle\ra{#5}}{\la{#5}}}^{#3}_{#4} \;
  {\mbox{\raise 1.3 mm\hbox{$#2$}}}}

\newtheorem{thm}{Theorem}

\newtheorem{lem}[thm]{Lemma}

\theoremstyle{definition}
\newtheorem{dfn}[thm]{Definition}
\newtheorem{exa}[thm]{Example}

\title{
A deficiency-based approach to parametrizing positive equilibria of biochemical reaction systems
}

\author{
Matthew Johnston \thanks{Department of Mathematics, San Jose State, One Washington Square, San Jose, CA, USA 95192}
\and
Stefan M\"uller \thanks{Faculty of Mathematics, University of Vienna, Austria}
\and
Casian Pantea \thanks{Department of Mathematics, West Virginia University, Morgantown, WV, USA 26506}
}

\begin{document}

\maketitle

\begin{abstract}

We present conditions which guarantee a parametrization of the set of positive equilibria of a {\em generalized} mass-action system. Our main results state that (i) if the underlying generalized chemical reaction network has an {\em effective deficiency} of zero, then the set of positive equilibria coincides with the parametrized set of complex-balanced equilibria and (ii) if the network is weakly reversible and has a {\em kinetic deficiency} of zero, then the equilibrium set is nonempty and has a positive, typically rational, parametrization. Via the method of network translation, we apply our results to {\em classical} mass-action systems studied in the biochemical literature, including the EnvZ-OmpR and shuttled WNT signaling pathways. 
A parametrization of the set of positive equilibria of a (generalized) mass-action system is often a prerequisite for the study of multistationarity and allows an easy check for the occurrence of absolute concentration robustness (ACR), as we demonstrate for the EnvZ-OmpR pathway.

\end{abstract}




\section{Introduction}
\label{sec:intro}

Networks of biochemical reactions can be represented as directed graphs where the vertices are combinations of interacting species (so-called complexes) and the edges are the reactions. Under suitable assumptions, such as spatial homogeneity and sufficient dilution, the networks follow mass-action kinetics and give rise to a system of polynomial ordinary differential equations in the species concentrations.

The mathematical study of positive equilibria of mass-action systems is important for establishing the uniqueness of equilibria in invariant regions of the state space (so-called compatibility classes) or, conversely, for establishing the capacity for multistationarity (for example, in models of biological switches). Such analysis, however, is challenging due to the high-dimensionality of the dynamical system, the significant nonlinearities, and the number of (unknown) parameters. Recent work has consequently focused on developing network-based methods for parameterizing the set of positive equilibria. Monomial parametrizations (Laurent monomials) have been established in \cite{C-D-S-S,M-D-S-C,MR2014,J2014}, while rational parametrizations have been constructed in \cite{T-G,G-H-R-S,MD2017}. Based on parametrizations, the uniqueness of positive equilibria has been analyzed \cite{C-F-R,MR2012,M-F-R-C-S-D,Conradi2017}, and regions for multistationarity (in the space of rate constants) have been identified for specific models, such as phosphorylation networks \cite{H-F-C,C-F-M-W2016}.

In this paper, we develop a method for explicitly constructing positive, typically rational, parametrizations of the set of positive equilibria for a broad class of biochemical reaction networks. Our approach is based on an extension of deficiency theory, the concept of generalized mass-action systems, and the method of network translation. The deficiency of a chemical reaction network was introduced in \cite{F1,H} in the context of sufficient conditions for weakly reversible networks to have complex-balanced equilibria \cite{H-J1}. The notions of deficiency and complex balancing were subsequently extended to generalized mass-action systems in \cite{MR2012,MR2014}. Thereby the kinetic complex determining the reaction rate was allowed to differ from the (stoichiometric) complex determining the reaction vector. Finally, the method of network translation was introduced in \cite{J2014}, in order to relate a mass-action system to a generalized mass-action system that is dynamically equivalent, but has a different network structure. In particular, the translated network might be weakly reversible (even when the original network is not) and have a lower deficiency.

A generalized mass-action system for which the underlying network is weakly reversible and has deficiency zero is known to have an equilibrium set with a monomial parametrization \cite{MR2012,MR2014,J2014}. In this paper, we extend this framework to construct positive parametrizations for a significantly wider class of generalized networks. To this end, we introduce a new notion of deficiency called \emph{effective deficiency} based on the \emph{condensed network} of the generalized network. Our main results state that, if a weakly reversible generalized network has an effective deficiency and kinetic deficiency of zero, then the corresponding generalized mass-action system permits a positive parametrization of the set of positive equilibria. This parametrization can be computed by linear algebra techniques and does not require tools from algebraic geometry such as Gr\"{o}bner bases. 
Via network translation, we can apply our results to a broad class of mass-action systems.

For example, consider the following two-component signaling system, which is adapted from a histidine kinase example in \cite{C-F-M-W2016}:
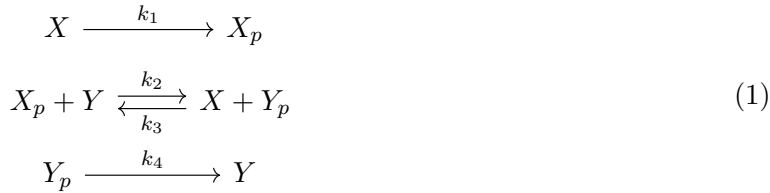
\begin{equation}
\label{intro-example}
\begin{tikzcd}
X \arrow[r, "k_1"] & X_p \\[-0.15in]
X_p + Y \arrow[r,  yshift=+0.5ex,"k_2"] & X + Y_p \arrow[l,  yshift=-0.5ex,"k_3"]  \\[-0.15in]  
Y_p \arrow[r, "k_4"] & Y
\end{tikzcd}
\end{equation}
Thereby $X$ is a histidine kinase, $Y$ is a response regulator, and $p$ is a phosphate group.
The network is not weakly reversible and has deficiency one. Via network translation, the system~\eqref{intro-example} corresponds to the following generalized mass-action system:
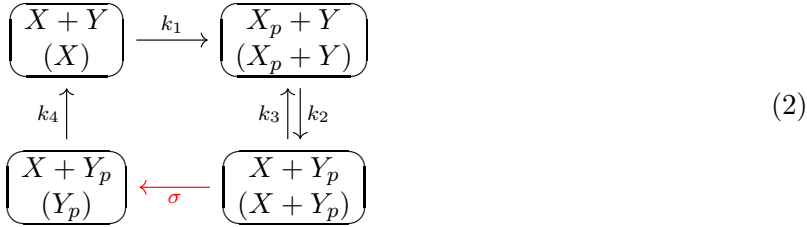
\begin{equation}
\label{example2}
\begin{tikzcd}
\ovalbox{$\begin{array}{c} X + Y \\ (X) \end{array}$} \arrow[r,"k_1"] & 
\ovalbox{$\begin{array}{c} X_p + Y \\ (X_p+Y) \end{array}$} \arrow[d,xshift=+0.5ex,"k_2"] \\
\ovalbox{$\begin{array}{c} X + Y_p \\ (Y_p) \end{array}$} \arrow[u,"k_4"] & 
\ovalbox{$\begin{array}{c} X + Y_p \\ (X+Y_p) \end{array}$} \arrow[u,xshift=-0.5ex,"k_3"] \arrow[l,color=red,"\fp"]
\end{tikzcd}
\end{equation}
Thereby we put a box at each vertex of the graph
with the \emph{stoichiometric complex} at the top and the \emph{kinetic complex} (in brackets) at the bottom. 
The red arrow corresponds to a \emph{phantom edge}, that is, an edge which connects identical stoichiometric complexes. Phantom edges do not contribute to the associated system of ordinary differential equations and hence can be labeled arbitrarily. Thus the edge label $\fp > 0$ can be considered a free parameter.

Now, the network~\eqref{example2} is weakly reversible and, as it turns out, it has an \emph{effective deficiency} of zero and a \emph{kinetic deficiency} of zero. Our main results guarantee that the set of positive equilibria has a positive parametrization and, in fact, constructively yield the following rational parametrization:
\begin{equation}
\label{intro:param}
\left \{ \quad
\begin{array}{ll} 
x = \displaystyle{\frac{k_4}{\fp}} , \quad 
& x_p = \displaystyle{\frac{k_1(k_3+\fp)k_4}{k_2\fp^2\fpp}} , \\ 
y = \fpp , 
& y_p = \displaystyle{\frac{k_1}{\fp}} ,\end{array} 
\right.
\end{equation}
where $\fp, \fpp > 0$. Note that the `rate constant' $\fp>0$ in the network~\eqref{example2} appears explicitly in the parametrization~\eqref{intro:param}. Importantly, the construction of \eqref{intro:param} via Theorem~\ref{main2} depends on efficient methods from linear algebra such as generalized inverses. Our algorithm therefore represents a significant computational advantage over algebraic geometry methods such as Gr\"{o}bner bases. 

\medskip

The paper is organized as follows. In Section~\ref{sec:background}, we review the relevant terminology regarding generalized chemical reaction networks and introduce several new notions, including effective and phantom edges, parametrized sets of equilibria, condensed networks, and effective deficiency. 
In Section~\ref{sec:results}, we present the crucial Lemma~\ref{cones} and the main results of the paper, Theorems~\ref{main} and~\ref{main2}. In Section~\ref{sec:applications}, we discuss the method of network translation, which allows us to apply the results of Section~\ref{sec:results} to networks studied in the biochemical literature, such as the EnvZ-OmpR and shuttled WNT signaling pathways.  In the EnvZ-OmpR example, our parametrization immediately implies the occurrence of absolute concentration robustness (ACR). In Section~\ref{sec:conclusions}, we summarize our findings and present avenues for future work.

\medskip

Throughout the paper, we use the following notation:
\begin{itemize}
\item
\emph{Vector logarithm:} for $v \in \RR^{n}_{>0}$, $\ln v = (\ln v_1, \ldots, \ln v_n)^T \in \RR^{n}$.
\item
\emph{Vector exponential:} for $v \in \RR^{n}$, $e^v = (e^{v_1}, \ldots, e^{v_n})^T \in \RR^{n}$.
\item
\emph{Vector powers:} for $v \in \RR^n_{>0}$ and $w \in \RR^{n}$, $v^w = \prod_{i=1}^n v_i^{w_i} \in \RR$.
\item
\emph{Matrix powers:} for $v \in \RR^{n}_{>0}$ and $A \in \RR^{n \times m}$, $v^{A^T} = (\prod_{i=1}^n v_i^{A_{i1}},\ldots,\prod_{i=1}^n v_i^{A_{im}})^T \in \RR^m$.
\item
\emph{Hadamard product:} for $v, w \in \RR^{n}$, $v \circ w = (v_1w_1, \ldots, v_nw_n)^T \in \RR^{n}$.
\end{itemize}

\section{Mathematical framework}
\label{sec:background}

We give a brief introduction to the relevant terminology regarding generalized chemical reaction networks (which include classical chemical reaction networks). In particular, we distinguish between effective and phantom edges and introduce parametrized sets of equilibria. Further, we define condensed networks and the notion of effective deficiency. Finally, we introduce the helpful concept of \vd networks.

\subsection{Generalized mass-action systems}
\label{sec:gcrn}

A directed graph $G=(V,E)$ is given by a set of vertices $V = \{ 1, \ldots, m \}$ and a set of edges $E \subseteq V \times V$. We denote an edge $e=(i,j)\in E$ by $i \to j$ to emphasize that is directed from the source $i$ to the target $j$. We additionally define the set of \emph{source vertices} $V_s = \{ i \mid i \to j \in E \}$, that is, the set of vertices that appear as the source of some edge.
We call the connected components of a graph \emph{linkage classes} and the strongly connected components \emph{strong linkage classes}. If linkage classes and strong linkage classes coincide, we call the graph \emph{weakly reversible}. 

A generalized chemical reaction network is essentially a graph with two embeddings of the vertices in $\RR^n$. The notion was introduced in \cite{MR2012} and extended to the present form in \cite{MR2014}.

\begin{dfn}
\label{def:gcrn}
A \emph{generalized chemical reaction network} (GCRN) $(G,y,\tilde{y})$ is given by a directed graph $G=(V,E)$ without self-loops and two maps $y \colon V \to \RR^n$ and $\tilde{y} \colon V_s \to \RR^n$.
Thereby, $G$ is called the abstract reaction graph,
and $y(i), \tilde y(i) \in \RR^n$ are called the stoichiometric and kinetic-order complexes, respectively, assigned to vertex $i$.
\end{dfn}

In contrast to a classical chemical reaction network (see below), a GCRN has two complexes associated to each vertex. Thereby, the maps $y$ and $\tilde{y}$ are not required to be injective, and the same stoichiometric or kinetic-order complex may be assigned to several vertices.

When considering examples, we represent complexes $y, \tilde y \in \RR^n$ as formal sums of \emph{species} (often $\{X_1, X_2, \ldots, X_n \}$). The components of the complexes correspond to the coefficients in the sums, e.g., $y = (1,0,1,0,\ldots,0)$ is represented as $y = X_1 + X_3$.

\begin{dfn}
\label{def:gmas}
A {\em generalized mass-action system} (GMAS) $(G_k,y,\tilde{y})$ is given by a GCRN $(G,y,\tilde{y})$ with $G=(V,E)$ together with edge labels $k \in \RR^E_{>0}$, resulting in the labeled directed graph $G_k$. That is, every edge $i \to j \in E$ is labeled with a rate constant $k_{i \to j} \in \RR_{>0}$.
\end{dfn}

The ODE system associated with a GMAS is given by
\begin{equation} \label{gmas}
\dd{x}{t} = f^{G}_k(x) = \sum_{i \to j \in E} k_{i \to j} \, x^{\tilde y(i)} \, (y(j) - y(i)).
\end{equation}
We can rewrite the right-hand side of the ODE as
\begin{equation} \label{gmas_fk}
f^{G}_k(x) = Y I_E \, \text{diag}(k) (I_E^s)^T \, x^{\tilde Y} = Y A^G_k \, x^{\tilde Y} ,
\end{equation}
where $Y, \tilde Y \in \RR^{n \times V}$ are the matrices of stoichiometric and kinetic complexes, respectively, $I_E, I_E^s \in \RR^{V \times E}$ are the {\em incidence} and {\em source} matrices of the graph $G$, and $A^G_k \in \RR^{V \times V}$ is the resulting {\em Laplacian matrix} of the labeled directed graph $G_k$. 
Note that columns $\tilde y^j$ of $\tilde Y$ corresponding to non-source vertices $j \not\in V_s$ can be chosen arbitrarily since the corresponding columns $(I^s_E)^j$ of $I^s_E$ and hence the columns $(A_k^G)^j$ of $A_k^G$ are zero vectors. 

Notably, the change over time~\eqref{gmas} lies in the {\em stoichiometric subspace} $S = \im (Y I_E)$,
which suggests the definition of a {\em stoichiometric compatibility class} $(c' + S) \cap \RR^n_{\ge 0}$ with $c' \in \RR^n_{\ge 0}$. The {\em stoichiometric deficiency} is defined as $\delta = \dim (\ker Y \cap \im I_E)$. Equivalently, $\delta = m - \ell - s$, where $m = |V|$ is the number of vertices, $\ell$ is the number of linkage classes of $G$,
and $s = \dim S$ is the dimension of the stoichiometric subspace (for example, see \cite{J2014}). If $V=V_s$, that is, if every vertex is a source, we additionally define the {\em kinetic-order subspace} $\tilde S = \im ( \tilde Y I_E)$ and the {\em kinetic deficiency} $\tilde \delta = 
\dim (\ker \tilde Y \cap \im I_E)$. Equivalently, $\tilde \delta = m - \ell - \tilde s$, where $\tilde s = \dim \tilde S$ is the dimension of the kinetic-order subspace.

\begin{exa}
\label{example3}
Consider the GCRN $(G,y,\tilde y)$ with the abstract graph $G = (V,E)$ given by
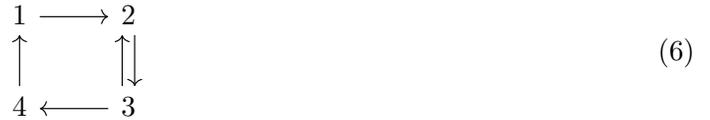
\begin{equation}
\label{example10}
\begin{tikzcd}
1 \arrow[r] & 2 \arrow[d,xshift=0.5ex] \\
4 \arrow[u] & 3 \arrow[u,xshift=-0.5ex] \arrow[l]
\end{tikzcd}
\end{equation}
and $y$ and $\tilde y$ defined by 
\[
y(1) = X_1+X_2, \, y(2) = X_2 + X_3, \, y(3) = y(4) = X_1 + X_4,
\]
and
\[
\tilde y(1) = X_1, \, \tilde y(2) = X_2 + X_3, \, \tilde y(3) = X_1 + X_4, \, \tilde y(4) = X_4 . 
\]
This generalized network has four vertices in one linkage class and is weakly reversible. It has a two-dimensional stoichiometric subspace ($s = 2$) and a three-dimensional kinetic-order subspace ($\tilde s = 3$). It follows that the stoichiometric deficiency is one ($\delta = 4-1-2 = 1$) while the kinetic deficiency is zero ($\tilde \delta = 4-1-3 = 0$). The corresponding generalized mass-action system $(G_k,y)$ gives rise to the following system of ODEs \eqref{gmas_fk}
\begin{equation}
\label{ex:dyn}
\dd{x}{t} = f^{G}_k(x) = Y A_k^G x^{\tilde Y} = \left[ \begin{array}{cccc} 1 & 0 & 1 & 1 \\ 1 & 1 & 0 & 0 \\ 0 & 1 & 0 & 0 \\ 0 & 0 & 1 & 1 \end{array} \right] \left[ \begin{array}{cccc} -k_{1 \to 2} & 0 & 0 & k_{4 \to 1} \\ k_{1 \to 2} & -k_{2 \to 3} & k_{3 \to 2} & 0 \\ 0 & k_{2 \to 3} & -k_{3 \to 2}-k_{3 \to 4} & 0 \\ 0 & 0 & k_{3 \to 4} & -k_{4 \to 1} \end{array} \right] \left[ \begin{array}{c} x_1 \\ x_2 x_3 \\ x_1 x_4 \\ x_4 \end{array} \right].
\end{equation}


Alternatively, we represent the abstract graph~\eqref{example10} and the maps $y$ and $\tilde y$ together in one graph,
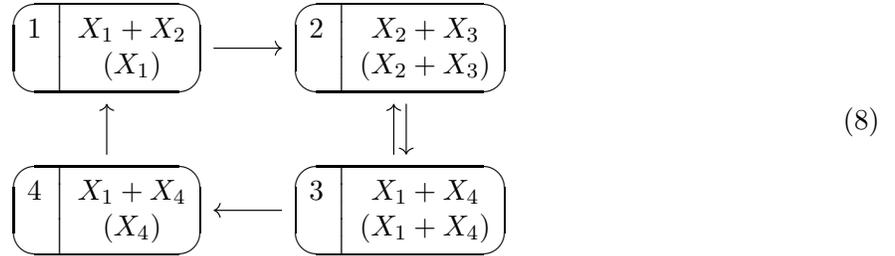
\begin{equation} \label{exampleXXX}
\begin{tikzcd}
\ovalbox{$\begin{array}{c} 1 \\ \\ \end{array}\Bigg\lvert \begin{array}{c} X_1 + X_2 \\ (X_1) \end{array}$} \arrow[r] & 
\ovalbox{$\begin{array}{c} 2 \\ \\ \end{array} \Bigg\lvert\begin{array}{c} X_2 + X_3 \\ (X_2+X_3) \end{array}$} \arrow[d,xshift=+0.5ex] \\
\ovalbox{$\begin{array}{c} 4 \\ \\ \end{array}\Bigg\lvert \begin{array}{c} X_1 + X_4 \\ (X_4) \end{array}$} \arrow[u] & 
\ovalbox{$\begin{array}{c} 3 \\ \\ \end{array}\Bigg\lvert \begin{array}{c} X_1 + X_4 \\ (X_1+X_4) \end{array}$} \arrow[u,xshift=-0.5ex] \arrow[l]
\end{tikzcd}
\end{equation}
\noindent where at each vertex we put a box with the vertex of the abstract graph (if required) on the left, the stoichiometric complex $y$ at the top, and the kinetic complex $\tilde y$ (in brackets) at the bottom. Note that the network~\eqref{example2} in the introduction is essentially the network~\eqref{exampleXXX} with specific interpretations of the species $X_1, X_2, X_3,$ and $X_4$ and of the edge labels $k_{i \to j}$ for $i \to j \in E$. 
\end{exa}

\subsection{Mass-action systems}

Classical chemical reaction networks and mass-action systems, which have been studied extensively in industrial chemistry and systems biology, can be considered as special cases of Definitions~\ref{def:gcrn} and~\ref{def:gmas}.

\begin{dfn}
\label{crn}
A \emph{chemical reaction network} (CRN) $(G,y)$ is a GCRN $(G,y,\tilde y)$ with $\tilde y = y$ and $y \colon V \mapsto \RR^n$ being injective. A mass-action system (MAS) $(G_k,y)$ is given by a CRN $(G,y)$, where $G=(V,E)$, together with a vector $k \in \RR^E_{>0}$, resulting in the labeled directed graph $G_k$.
\end{dfn}

\noindent The ODE system associated with a MAS is given by
\begin{equation} \label{mas}
\dd{x}{t} = f^{G}_k(x) := \sum_{i \to j \in E} k_{i \to j} \, x^{y(i)} \, (y(j) - y(i)) = Y I_E \, \text{diag}(k) (I^s_E)^T \, x^{Y} = Y A_k^G \, x^{Y}
\end{equation}
where the matrices $Y$, $I_E$, $I^s_E$, and $A_k^G$ are as in~\eqref{gmas_fk}.

For a CRN, the stoichiometric and kinetic-order subspaces coincide (i.e.~$S = \tilde S$), and the stoichiometric and kinetic deficiencies are the same (i.e.~$\delta = \tilde \delta$). In fact, the \emph{deficiency} $\delta = \dim( \ker Y \cap \im I_E) = m - \ell - s$ was introduced first in \cite{F1,H} in the context of complex-balanced mass-action systems \cite{H-J1}. It has been studied extensively since then \cite{Fe2,F2,F3,Sh-F}.

In a CRN, the map $y$ is unique and vertices and complexes are in one-to-one correspondence.
It is typical to write the reaction graph $G$ with the complexes as vertices.

\begin{exa}
Recall the CRN $(G,y)$~\eqref{intro-example} from the introduction, where the species and rate constants $k \in \RR_{>0}^E$ have been relabeled as follows:
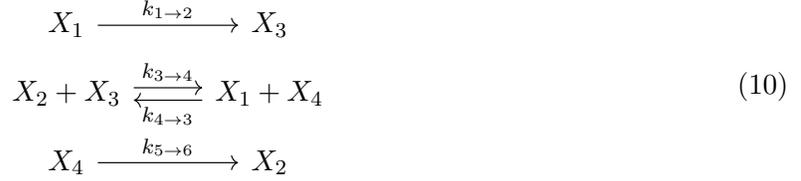
\begin{equation}
\label{example237}
      \begin{tikzcd}
X_1 \arrow[r, "k_{1 \to 2}"] & X_3 \\[-0.15in]
X_2 + X_3 \arrow[r,  yshift=+0.5ex,"k_{3 \to 4}"] & X_1 + X_4 \arrow[l,  yshift=-0.5ex,"k_{4 \to 3}"]  \\[-0.15in]
X_4 \arrow[r, "k_{5 \to 6}"] & X_2
      \end{tikzcd}
\end{equation}
The CRN~\eqref{example237} has six vertices in three linkage classes, and is not weakly reversible. It has a stoichiometric subspace of dimension two ($s=2$), and hence its deficiency is one ($\delta = 6 - 3 - 2 = 1$).

The system of ODEs~\eqref{mas} associated with the MAS $(G_k,y)$~\eqref{example237} is
\begin{equation}
\label{ex:example237}
\left\{ \quad
\begin{aligned}
\dot{x}_1 & = -k_{1 \to 2} x_1 + k_{3 \to 4} x_2 x_3 - k_{4 \to 3} x_1 x_4 \\
\dot{x}_2 & = -k_{3 \to 4} x_2 x_3 + k_{4 \to 3} x_1 x_4 + k_{5 \to 6} x_4 \\
\dot{x}_3 & = k_{1 \to 2} x_1 - k_{3 \to 4} x_2 x_3 + k_{4 \to 3} x_1 x_4 \\
\dot{x}_4 & = k_{3 \to 4} x_2 x_3 - k_{4 \to 3} x_1 x_4 - k_{5 \to 6} x_4.
\end{aligned}
\right.
\end{equation}
Notably, after expanding and relabeling the rate constants in the ODE system~\eqref{ex:dyn} arising from the GCRN~\eqref{example10}, the ODE system~\eqref{ex:example237} arising from the CRN~\eqref{example237} coincides with \eqref{ex:dyn}. Results obtained by a structural analysis of the GCRN~\eqref{example10} will consequently hold for the CRN~\eqref{example237}. In particular, we will investigate existing methods for corresponding MASs and GMASs with equivalent dynamics,~\eqref{mas} and~\eqref{gmas}, respectively, in Section~\ref{sec:translation}. 
\end{exa}

\subsection{Effective and phantom edges and parametrized sets of equilibria}

For a GCRN, only edges $i \to j \in E$ with $y(j) \neq y(i)$ contribute to the right-hand side of the ODE~\eqref{gmas}. In Example~\ref{example3}, $y(3)=y(4)$, and hence the rate constant $k_{3 \to 4}$ does not appear in the ODEs~\eqref{ex:dyn}, even though $3 \to 4 \in E$. Consequently, we may partition the set of edges $E$ into the set of {\em effective edges}
\[E^* = \{ i \to j \in E \mid y(i) \neq y(j) \}\]
and the set of {\em phantom edges}
\[E^0 = \{ i \to j \in E \mid y(i) = y(j) \}.\]
Obviously, $E^* \cap E^0 = \emptyset$ and $E = E^* \cup E^0$. 
For a vector $k \in \RR^E_{>0}$,
we define $k^* = k_{E^*} \in \RR^{E^*}_{>0}$ and $k^0 = k_{E^0} \in \RR^{E^0}_{>0}$ so that $k=(k^*,k^0)$. Further, we introduce the effective graph $G^* = (V,E^*)$.

From~\eqref{gmas} it follows that
\begin{equation}
\label{gmas_simplified}
f^{G}_k(x) = \sum_{i \to j \in E} k_{i \to j} \, x^{\tilde y(i)} \, (y(j) - y(i)) = \sum_{i \to j \in E^*} k_{i \to j} \, x^{\tilde y(i)} \, (y(j) - y(i)) = f^{G^*}_{k^*}(x) .
\end{equation}
That is, the GMAS $(G_k,y,\tilde y)$ gives rise to the same system of ODEs as the GMAS $(G^*_k,y,\tilde y)$, involving the effective graph $G^*$. In particular, the dynamics does not depend on $k^0$.
From~\eqref{gmas_simplified} and~\eqref{gmas_fk} it follow that 
\begin{equation}
\label{gmas_s}
f^{G}_k(x) = f^{G}_{(k^*,k^0)}(x) = f^{G}_{(k^*,\fp)}(x) = Y A^{G}_{(k^*,\fp)} \, x^{\tilde Y} ,
\end{equation}
for arbitrary $\fp \in \RR^{E^0}_{>0}$.
That is, we may replace the rate constants~$k^0$ by arbitrary parameters~$\fp$.

For a GMAS $(G_k,y,\tilde{y})$,
the set of positive equilibria is given by
\[
\eq^G_{k} := \{ x \in \RR^n_{>0} \mid f^{G}_{k}(x) = 0 \} ,
\]
while the set of positive {\em complex-balanced} equilibria (CBE) is given by
\[
Z^G_k := \{ x \in \RR^n_{>0} \mid A^G_k \, x^{\tilde Y} = 0 \} \subseteq \eq^G_{k} .
\]
Note that $\eq^G_k = \eq^{G^*}_{k^*}$, and hence the equilibrium set $\eq^G_k$ depends on $k^*$, but not on $k^0$, while $Z_k$ depends on both $k^*$ and $k^0$. 

Equation~\eqref{gmas_s} motivates another definition. For an arbitrary parameter $\fp \in \RR^{E^0}_{>0}$, we consider
\[
Z^G_{(k^*,\fp)} := \{ x \in \RR^n_{>0} \mid A^G_{(k^*,\fp)} \, x^{\tilde Y} = 0 \} \subseteq \eq^G_k ,
\]
which is the set of positive CBE of the GMAS $(G_{(k^*,\fp)},y,\tilde{y})$.
The {\em parameterized} set of positive CBE (PCBE) is given by
\[
\bar Z^G_k := \bigcup_{\fp \in \RR^{E^0}_{>0}} Z^G_{(k^*,\fp)} \subseteq \eq^G_{k} ,
\]
thereby varying over all $\fp \in \RR^{E^0}_{>0}$.

For a GMAS $(G_k,y,\tilde y)$, the set $\bar Z^G_k$ need not coincide with the set $\eq^G_{k}$. In our main results, however, we give conditions on the underlying GCRN $(G,y,\tilde y)$ such that $\eq^G_k = \bar Z^G_k$ (Theorem~\ref{main}), and also conditions under which a positive parametrization of $\bar Z^G_k$ can be constructed (Theorem~\ref{main2}). 

\begin{exa} 
\label{example4}
Recall the GCRN~\eqref{example10} from Example~\ref{example3}. 
The edge set $E$ can be partitioned into effective edges $E^* = \{ 1 \to 2, \, 2 \to 3, \, 3 \to 2, \, 4 \to 1 \}$ and phantom edges $E^0 = \{ 3 \to 4 \}$. The equilibrium set $\eq_{k}$ is determined by setting the right-hand sides of the ODEs~\eqref{ex:dyn} to zero, whereas the set $Z_k$ of CBE is determined by the Laplacian matrix,
\begin{equation}
\label{ex:cb}
A_k^Gx^{\tilde Y} = \left[ \begin{array}{cccc} -k_{1 \to 2} & 0 & 0 & k_{4 \to 1} \\ k_{1 \to 2} & -k_{2 \to 3} & k_{3 \to 2} & 0 \\ 0 & k_{2 \to 3} & -k_{3 \to 2}-k_{3 \to 4} & 0 \\ 0 & 0 & k_{3 \to 4} & -k_{4 \to 1} \end{array} \right] \left[ \begin{array}{c} x_1 \\ x_2 x_3 \\ x_1 x_4 \\ x_4 \end{array} \right] = \left[ \begin{array}{c} 0 \\ 0 \\ 0 \\ 0 \end{array} \right].
\end{equation}
Note that these equations depend on the rate constant $k_{3 \to 4}$, even though it does not appear in the ODEs~\eqref{ex:dyn}. By replacing $k_{3 \to 4}$ with an arbitrary parameter $\fp$ in~\eqref{ex:cb}, we obtain the new set of CBE $Z_{(k^*,\fp)}$. The set $\bar Z_{k}$ of PCBE is obtained by varying over all $\fp \in \RR_{>0}$. A constructive method for solving systems like \eqref{ex:cb} for the concentrations $x_i$ will be discussed in Section \ref{sec:kdzt}.
\end{exa}

\subsection{Condensed networks and effective deficiency}

We now consider auxiliary networks with special properties. First, we introduce a network that condenses stoichiometrically identical vertices and thereby removes phantom edges.

\begin{dfn}
\label{def:condensed}
For the GCRN $(G,y,\tilde y)$, we define the {\em condensed} CRN $(G',y')$ given by the digraph $G' = (V', E')$, where
\begin{enumerate}
\item
$V' = \{ [i] \mid i \in V \}$ with $[i] = \{ j \in V \mid y(j)=y(i) \}$ for $i \in V$ and
\item
$E' = \{ [i] \to [j] \mid i \to j \in E^* \}$,
\end{enumerate}
and the map $y' \colon V' \to \RR^n, \, y'([i]) = y(i)$. 
\end{dfn}

For the GCRN $(G,y,\tilde y)$,
we define the {\em effective} deficiency as the deficiency of the condensed CRN $(G',y')$,
\[
\delta' = \dim(\ker Y' \cap \im I_{E'})
\]
with the incidence matrix $I_{E'} \in \RR^{V' \times E'}$ and the matrix of complexes $Y' \in \RR^{n \times V'}$, as defined after \eqref{gmas_fk} in Section~\ref{sec:gcrn}.
Equivalently, $\delta' = m' - \ell' - s$, where $m' = |V'|$ is the number of vertices and $\ell'$ is the number of linkage classes of $G'$. Thereby, we use $S' = \im (Y' I_{E'}) = \im (Y I_E) = S$
and hence $s' = \dim(S') = \dim(S) = s$.

Finally, we define a section $\rho \colon V' \to V$, assigning to each equivalence class $[i] \in V'$ a representative vertex $\rho([i]) \in [i]$,
that is, we define
a set of representative vertices $V^* = \{ \rho([i]) \mid [i] \in V' \} \subseteq V$, containing exactly one representative vertex from each equivalence class.


\begin{exa}
\label{ex:1}
Recall the GCRN~\eqref{example10} from Examples~\ref{example3} and~\ref{example4}, in particular, that $y(3)=y(4)=X_1+X_4$. Hence we have the equivalence classes
\[
[1] = \{1\}, \; [2] = \{2\}, \; [3] = [4] = \{3,4\}.
\]
For the GCRN, we obtain the condensed CRN $(G',y')$, in particular, the graph~$G'$
\begin{equation} 
\label{condensed}
\begin{tikzcd}
\{1\} \arrow[r] & \{2\} \arrow[d,xshift=0.5ex] \\
& \{3,4\} \arrow[u,xshift=-0.5ex] \arrow[lu]
\end{tikzcd}
\end{equation}
and the map $y'$
with $y'(\{1\}) = X_1$, $y'(\{2\}) = X_2 + X_3$, and $y'(\{3,4\}) = X_1 + X_4$. Note that we do not associate kinetic complexes to the vertices of the condensed graph. The deficiency of~\eqref{condensed} is $\delta = 3-1-2=0$, that is, the effective deficiency of the GCRN~\eqref{example10} is $\delta' = 0$.
\end{exa}

\subsection{\vd networks}

Second, we introduce a class of GCRNs which is helpful for constructing a positive para\-metri\-zation of the equilibrium set. 

\begin{dfn}
\label{def:vstar}
Let $(G,y,\tilde y)$ be a GCRN with $G=(V,E)$ and condensed CRN $G'=(V',E')$.
Further, let $V^* \subseteq V$ be a set of representative vertices. (That is, there is a section $\rho \colon V' \to V$ such that $V^* = \{ \rho([i]) \mid [i] \in V' \}$.)
We say that $(G,y,\tilde y)$ is \emph{\vd} if
\[
j \to i \in E^*
\quad \text{implies} \quad 
i \in V^*, 
\quad \text{that is,} \quad 
i = \rho([i]),
\]
and
\begin{align*}
E^0 
= \{ i \to j \mid i \in V^* , \, j \in [i]  \setminus \{i\} \} ,
\quad \text{that is,} \quad
E^0 = \{ \rho([i]) \to j \mid [i] \in V', \, j \in [i] \setminus \{\rho([i])\} \} .
\end{align*}
\end{dfn}




A GCRN being \vd guarantees that effective edges (those between equivalence classes $[i]$) enter at the representative vertex $\rho([i]) \in V^*$, and that phantom edges (those within an equivalence class $[i]$) lead from $\rho([i])$ to the other vertices in the class. The representative vertices $\rho([i]) \in V^*$ may be thought of as the hubs of the representative equivalence classes through which all directed paths must travel.

The class of \vd GCRNs may seem restrictive. The following result, however, guarantees that, for every GMAS, there is a dynamically equivalent GMAS which is \vd, that is, the associated ODEs agree, cf. \eqref{gmas}. This will be instrumental in applications, cf.~Section~\ref{sec:applications}.



\begin{lem}
\label{lemma:equiv2}
Let $(G,y,\tilde y)$ be a GCRN with $G=(V,E)$ and
representative vertex set $V^* \subseteq V$, 
and let $k \in \RR_{> 0}^{E}$ be a rate vector. 
Then there is a GCRN $(\rd G,y,\tilde y)$ with $\rd G=(V, \rd E)$ that is \vd and a rate vector $\rd k \in \RR_{>0}^{\rd E}$ such that 
the GMASs $(G_k,y,\tilde y)$ and $(\rd G_{\rd k},y,\tilde y)$ are dynamically equivalent,
that is, the associated ODEs agree, cf.~\eqref{gmas}. 
\end{lem}

\begin{proof}
First we define the set $\rd E^0 = \{ i \to j \mid i \in V^* , \, j \in [i] \setminus \{i\} \}$ and associate an arbitrary $\rd k_{i \to j} > 0$ to each edge $i \to j \in \rd E^0$. Then we define the set $\rd E^* = \rd E^*_1 \cup \rd E^*_2$ as follows:
\begin{enumerate}
\item
If $i \to j \in E^*$ and $j \in V^*$, then $i \to j \in \rd E^*_1$ and $\rd k_{i \to j} = k_{i \to j}$.
\item
If $i \to j \in E^*$ and $j \not\in V^*$, then $i \to \rho([j]) \in \rd E^*_2$ and $\displaystyle{\rd k_{i \to \rho([j])} = \sum_{j' \in [j] \setminus \{\rho([j])\}} k_{i \to j'}}.$
\end{enumerate}

Now we consider the GCRN $(\rd G, y, y')$ with $\rd G = (V, \rd E)$ and $\rd E = \rd E^0 \cup \rd E^*$, which is \vd by construction. With the vector $\rd k \in \RR_{> 0}^{\rd E}$ constructed above, we have
\[
\begin{split} 
f^G_k(x) 
& = \sum_{\substack{i \to j \in E^* \\ j \in V^*}} k_{i \to j} \, x^{\tilde y(i)} \, (y(j) - y(i)) + \sum_{\substack{i \to  j' \in E^* \\ j' \not\in V^*}} k_{i \to j'} \, x^{\tilde y(i)} \, (y(j')) - y(i)) \\
& = \sum_{i \to j \in \rd E^*_1} \rd k_{i \to j} \, x^{\tilde y(i)} \, (y(j) - y(i)) + \sum_{i \to \rho([j]) \in \rd E^*_2} \sum_{j' \in [j] \setminus \{\rho([j])\}} k_{i \to j'} \, x^{\tilde y(i)} \, (y(j')) - y(i)) \\
& = \sum_{i \to j \in \rd E^*_1} \rd k_{i \to j} \, x^{\tilde y(i)} \, (y(j) - y(i)) + \sum_{i \to  \rho([j]) \in \rd E^*_2} \rd k_{i \to  \rho([j])} \, x^{\tilde y(i)} \, (y(\rho([j])) - y(i)) \\
&= f^{\rd G}_{\rd k}(x) ,
\end{split}
\]
where we have omitted the edge sets $E^0$ and $\tilde E^0$ according to~\eqref{gmas_simplified}.
\end{proof}

\begin{exa}
Recall the GCRN~\eqref{example10} from Examples~\ref{example3},~\ref{example4} and~\ref{ex:1}. Since $y(3) = y(4)$ and hence $[3]=[4]=\{3,4\}$,
we have two possible sections $\rho$, that is, two possible sets of representative vertices $V^*$, namely, $V_1^* = \{1, 2, 3\}$ and $V_2^* = \{1, 2, 4\}$.

For the set $V_1^*$, all edges in~\eqref{example10} enter $\{3,4\}$ at $3 = \rho(\{3,4\})$, and the phantom edge $3 \to 4$ leads from $3 = \rho(\{3,4\})$ to $4$. Hence \eqref{example10} is $V_1^*$-directed.

For the set $V_2^*= \{1, 2, 4\}$, the edge $2 \to 3$ in~\eqref{example10} leads to $3 \neq \rho(\{3,4\})$. We therefore replace it by the edge $2 \to 4$ with $4 = \rho(\{3,4\})$. Further, we replace the phantom edge $3 \to 4$ by the phantom edge $4 \to 3$. This construction yields the following $V_2^*$-directed GCRN $(\rd G,y, \tilde y)$:
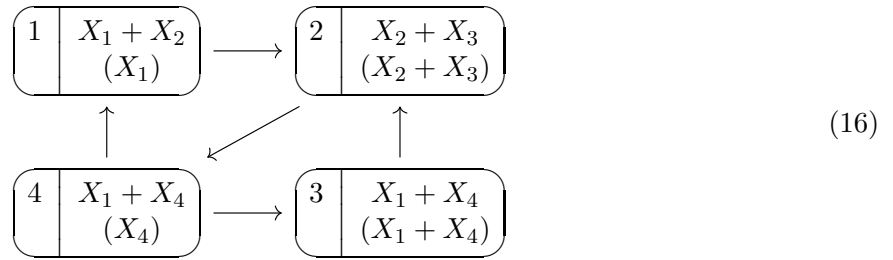
\begin{equation}
\label{example32}
\begin{tikzcd}
\ovalbox{$\begin{array}{c} 1 \\ \\ \end{array}\Bigg\lvert \begin{array}{c} X_1 + X_2 \\ (X_1) \end{array}$} \arrow[r] & \ovalbox{$\begin{array}{c} 2 \\ \\ \end{array} \Bigg\lvert\begin{array}{c} X_2 + X_3 \\ (X_2+X_3) \end{array}$} \arrow[dl] \\
\ovalbox{$\begin{array}{c} 4 \\ \\ \end{array}\Bigg\lvert \begin{array}{c} X_1 + X_4 \\ (X_4) \end{array}$} \arrow[u]  \arrow[r]& \ovalbox{$\begin{array}{c} 3 \\ \\ \end{array}\Bigg\lvert \begin{array}{c} X_1 + X_4 \\ (X_1+X_4) \end{array}$}  \arrow[u]
\end{tikzcd}
\end{equation}
The corresponding rate vector $\rd k \in \RR_{> 0}^{\tilde E}$ is $\rd k_{1 \to 2} = k_{1 \to 2}$, $\rd k_{2 \to 4} = k_{2 \to 3}$, $\rd k_{3 \to 2} = k_{3 \to 2}$, $\rd k_{4 \to 1} = k_{4 \to 1}$, and $\rd k_{4 \to 3} = k_{3 \to 4}$. Hence $f_k^G = f_{\rd k}^{\rd G}$, cf.~\eqref{gmas}.
\end{exa}


\section{Main results}
\label{sec:results}


In Section~\ref{sec:edzt}, we consider GCRNs with an effective deficiency of zero ($\delta' = 0$)
and present Theorem~\ref{main}, stating that the set of positive equilibria coincides with the parametrized set of complex-balanced equilibria (PCBE).
In Section~\ref{sec:kdzt}, we consider GCRNs with a kinetic deficiency of zero ($\tilde \delta = 0$) and higher ($\tilde \delta > 0$) and present Theorem~\ref{main2}, explicitly constructing the PCBE. 

\subsection{Effective deficiency}
\label{sec:edzt}

Lemma~\ref{cones} below is crucial for the proof of Theorem~\ref{main}. In the following, we write $\cone W \subseteq \RR^n$ for the polyhedral cone generated by the columns of the matrix $W \in \RR^{n \times m}$.

\begin{lem} \label{cones}
Let $(G,y,\tilde y)$ be a GCRN 
with $G=(V,E)$ and representative vertex set $V^* \subseteq V$.
In particular, let $(G,y,\tilde y)$ be \vd
and have effective deficiency $\delta'=0$. 
Then
\[ 
\ker Y \cap \cone I_{E^*} \subseteq \cone (-I_{E^0}).
\]
Moreover,
\[ 
\ker Y \cap \relint(\cone I_{E^*}) \subseteq \relint(\cone (-I_{E^0})) .
\]
\end{lem}

\begin{proof}
Let $v \in (\ker Y \cap \cone I_{E^*}) \subseteq \RR^V$, that is, $v = I_{E^*} \, x = \sum_{i \to j \in E^*} x_{i \to j} \, (e_j - e_i)$ with $x \in \RR^{E^*}_{\ge0}$ (nonnegative weights on the effective edges $E^*$) and
\begin{align*}
0 = Y v 
& = \sum_{i \to j \in E^*} x_{i \to j} \, (y(j) - y(i)) \\ & = \sum_{[i] \to [j] \in E'} \bigg( \sum_{\substack{i' \to j' \in  E^*: \\ i' \in [i], j' \in [j]}} x_{i' \to j'} \bigg) \, (y'([j]) - y'([i])) \\
&= Y' \sum_{[i] \to [j] \in E'} x'_{[i] \to [j]} \,(e_{[j]} - e_{[i]}) \\
&= Y' \, I_{E'} \, x' = Y' \, v' .
\end{align*} 
Thereby, $(G',y')$ with $G'=(V',E')$ is the corresponding condensed CRN
and $v' = I_{E'} \, x' \in \RR^{V'}$ with $x' \in \RR^{E'}_{\ge0}$.
Clearly,
$
v'_{[i]} = \sum_{i' \in [i]} v_{i'}
$
for $[i] \in V'$.

Now, $\delta'=\dim(\ker Y' \cap \im I_{E'})=0$ implies $v'=0$, that is,
\begin{align*}
0 = v'_{[i]} 
= \sum_{i' \in [i]} v_{i'}
\end{align*}
for $[i] \in V'$.
Using that $G$ is \vd, 
reconsider $v = I_{E^*} \, x \in \RR^V$ (the fluxes arising from the effective edges $E^*$). 
Let $i \in V^*$, that is, $i = \rho([i])$.
For $i' \in [i] \setminus \{i\}$,
\[
v_{i'} = - \sum_{i' \to j \in E^*} x_{i' \to j} ,
\]
whereas
\[
v_{i} = - \sum_{i' \in [i] \setminus \{i\}} v_{i'} .
\]
%
Now, choose $\tilde x \in \RR^{E^0}_{\ge0}$ (nonnegative weights on the phantom edges $E^0$) as
\begin{equation} \label{tilde_x}
\tilde x_{i \to i'} = \sum_{i' \to j \in E^*} x_{i' \to j} ,
\end{equation}
where $i' \in [i]$.
Then,
for $i' \in [i] \setminus \{i\}$,
\[
v_{i'} = - \tilde x_{i \to i'} ,
\]
whereas
\[
v_{i} = - \sum_{i' \in [i] \setminus \{i\}} v_{i'} = \sum_{i'\in[i] \setminus \{i\}} \tilde x_{i \to i'} = \sum_{i \to i' \in E^0} \tilde x_{i \to i'} .
\]
That is,
$- v = I_{E^0} \, \tilde x \in \RR^V$ (the fluxes arising from the phantom edges $E^0$),
and hence $v \in \cone (-I_{E^0})$.



Finally, let $v \in (\ker Y \cap \relint(\cone I_{E^*}))$, that is,
$v = I_{E^*} \, x$ for some $x \in \RR^{E^*}_{>0}$.
Then $v = -I_{E^0} \, \tilde x\in \relint(\cone(-I_{E^0}))$ with $\tilde x \in \RR^{E^0}_{>0}$ by~\eqref{tilde_x}.
\end{proof}

We now present the main result of this section, which gives conditions under which the equilibrium set $X_k^G$ coincides with the parametrized set of complex-balanced equilibria $\bar{Z}_k^G$.

\begin{thm} \label{main}
Let $(G,y,\tilde y)$ be a GCRN with effective deficiency $\delta'=0$. 
Further, let $(G,y,\tilde y)$ be \vd for a set of representative vertices $V^* \subseteq V$. 
Then, for the GMAS $(G_k,y,\tilde y)$, 
the set of positive equilibria agrees with the parametrized set of complex-balanced equilibria,
that is, $\eq^G_{k} = \bar Z^G_k$.
\end{thm}

\begin{proof}
Let $x \in \RR_{>0}^n$ be a positive equilibrium, that is, $x \in \eq^G_k$. 
Using $G^*=(V,E^*)$, $G^0=(V,E^0)$, and $\fp \in \RR^{E^0}_{>0}$,
we may write
\[
A^G_{(k^*,\fp)} \, x^{\tilde Y} = A^{G^*}_{k^{\ast}} \, x^{\tilde Y} + A^{G^0}_\fp x^{\tilde Y} ,
\]
cf.~\eqref{gmas}.
Now $x \in \eq^G_k = \eq^{G^*}_{k^*}$ implies $Y A^{G^*}_{k^*} x^{\tilde Y} = 0$ and hence
\[
A^{G^*}_{k^*} \, x^{\tilde Y} \in (\ker Y \cap \relint(\cone I_{E^{\ast}})) ,
\]
cf.~\eqref{gmas_fk}. 
Since $\delta' = 0$ and $(G,y,\tilde y)$ is \vd, we have $
A^{G^*}_{k^{\ast}} \, x^{\tilde Y} \in \relint(\cone(-I_{E^0})) ,
$
by Lemma~\ref{cones}.
That is,
\[
A^{G^*}_{k^{\ast}} \, x^{\tilde Y} = - \sum_{i \to j \in E^0} \alpha_{i \to j} \, (e_j - e_i)
\]
for some $\alpha \in \RR^{E^0}_{>0}$. 
On the other hand,
\[
A^{G^0}_\fp x^{\tilde Y} =  \sum_{i \to j \in E^0} \fp_{i \to j} \, x^{\tilde y(i)} \, (e_j - e_{i}).
\]
We choose $\fp_{i \to j} = \alpha_{i \to j} / x^{\tilde y(i)}$ for $i \to j \in E^0$ such that $A^{G^*}_{k^*} x^{\tilde Y} = - A^{G^0}_{\fp} x^{\tilde Y}$ and hence $A^G_{(k^*,\fp)} \, x^{\tilde Y} = 0$,
that is, $x \in Z^G_{(k^*,\fp)} \subseteq \bar Z^G_k$ so $\eq^G_{k} \subseteq \bar Z^G_k$. Since $\bar Z^G_k \subseteq \eq^G_{k}$ trivially, we have $\eq^G_{k} = \bar Z^G_k$.
\end{proof}

\subsection{Kinetic deficiency}
\label{sec:kdzt}

We fix the graph $G=(V,E)$ and omit the corresponding superscript,
that is, we write $A_k^G = A_k$, $Z^G_k = Z_k$, and $\bar Z_k^G = \bar Z_k$.
Recall that $x \in Z_k$ is equivalent to $x^{\tilde Y} \in \ker A_{k}$. Following \cite{J2014,MR2014}, we discuss $\ker A_k$. First, we introduce the vector of {\em tree constants} $K \in \RR^V_{>0}$ with entries 
\[
K_i = \sum_{(\mathcal V, \mathcal E) \in T_i} \prod_{i' \to j' \in \mathcal E} k_{i' \to j'} , \quad i \in V ,
\]
where $T_i$ is the set of directed spanning trees (of the respective linkage class) rooted at vertex~$i$. Clearly, the tree constants $K$ depend on the rate constants $k \in \RR^E_{>0}$, that is, $K = K(k)$.

For a weakly reversible GCRN,
\[
\ker A_k = \Span \{ v^1, \ldots, v^{\ell} \} 
\]
with nonnegative vectors $v^l \in \RR_{\geq 0}^n$ (for $l = 1, \ldots, \ell$) having support on the respective linkage class~$l$.
In particular, $v^l_i = K_i$ if vertex $i$ is in linkage class $l$ and $v^l_i = 0$ otherwise. Now, $x^{\tilde Y} \in \ker A_k$ if and only if
\[
x^{\tilde Y} = \sum_{l=1}^\ell \alpha_l \, v^l 
\]
with $\alpha_l > 0$. For any pair of vertices $i$ and $j$ in the same linkage class, we have
\[
\frac{x^{\tilde y(i)}}{K_{i}} = \frac{x^{\tilde y(j)}}{K_{j}}.
\]
Taking the logarithm gives
\begin{equation} \label{eq:loglin}
(\tilde y(i) - \tilde y(j))^T \ln x = \ln \left( \frac{K_{i}}{K_{j}} \right).
\end{equation}

Now we choose a spanning forest $F = (V,\mathcal E)$ for $G = (V,E)$, that is,
we choose spanning trees for all linkage classes. Note that $F$ contains the same vertices as $G$, but not the same edges. Also note that, in the following results and applications, the choice of the spanning tree is arbitrary.
Clearly, the spanning tree of linkage class~$l$ contains $m_l$ vertices and $m_l - 1$ edges. Hence, the spanning forest $F$ contains $m$ vertices and $m-\ell$ edges. We introduce the matrix $M = \tilde Y I_\mathcal{E} \in \RR^{n \times \mathcal E}$ whose $m-\ell$ columns are given by $\tilde y(j) - \tilde y(i)$ for $i \to j \in \mathcal E$. Correspondingly, we define the vector $\tc \in \RR^\mathcal{E}_{>0}$ whose $m-\ell$ entries are given by $\tc_{i \to j} = \frac{K_i}{K_j}$ for $i \to j \in \mathcal E$. As for $K$, we note that $\tc$ depends on $k$, that is, $\tc=\tc(k)$. 
Hence, we can write the system of equations~\eqref{eq:loglin} as
\begin{equation} \label{loglinear}
M^T \ln x = \ln \tc .
\end{equation}

Theorem~1 in \cite{MR2014} implies the following result.

\begin{thm} \label{main2}
Let $(G,y,\tilde y)$ be a GCRN that is weakly reversible, 
and let $(G_k,y,\tilde y)$ be a GMAS. Further, let $M \in \RR^{n \times \mathcal{E}}$ and $\tc=\tc(k)=\tc(k^*,k^0) \in \RR^\mathcal{E}_{>0}$ be defined as above, and let $H \in \RR^{n \times \mathcal{E}}$ be a generalized inverse of $M^T$ (that is, $M^T H M^T = M^T$). Finally, define $B \in \RR^{n \times (n - \tilde{s})}$ with $\im B = \ker M^T$ and $\ker B = \{0\}$, and $C \in \RR^{\mathcal{E} \times \tilde \delta}$ with $\im C = \ker M$ and $\ker C = \{0\}$. 
\begin{enumerate}
\item
If the kinetic deficiency is zero ($\tilde \delta = 0$), then $\bar Z_{k} \not= \emptyset$, in particular, $\bar Z_{k}$ has the positive parametrization
\begin{equation}
\label{eq:param}
\bar Z_{k} = \left\{ \tc(k^*,\fp)^{H^T} \circ \fpp^{B^T} \mid \fp \in \RR^{E^0}_{>0}, \, \fpp \in \RR^{n - \tilde s}_{>0} \right\}.
\end{equation}
\item
If the kinetic deficiency is positive ($\tilde \delta > 0$)
and the $\tilde \delta$ equations
\begin{equation}
\label{eq:condition}
\tc(k^*,k^0)^{C} = 1^{\tilde \delta \times 1}
\end{equation}
can be solved explicitly for $\tilde \delta$ components of $k^0 \in \RR^{E^0}_{>0}$ (in terms of $k^* \in \RR^{E^*}_{>0}$ and the remaining components of $k^0$),
that is,
if there exists an explicit function $h \colon \RR^{E^* \cup (E^0 \setminus \tilde E^0)}_{>0} \to \RR^{\tilde E^0}_{>0}$ with $\tilde E^0 \subseteq E^0$, $|\tilde E^0| = \tilde \delta$, and $k^0 = (\tilde k^0, \cdot) \in \RR^{(E^0 \setminus \tilde E^0) \cup \tilde E^0}_{>0}$ such that, for all $k^* \in \RR_{> 0}^{E^*}$ and $\tilde k^0 \in \RR_{> 0}^{E^0 \setminus \tilde E^0}$,
\[
\tc(k^*,(\tilde k^0, h(k^*,\tilde k^0)))^{C} = 1^{\tilde \delta \times 1} ,
\]
then $\bar Z_{k} \neq \emptyset$, and $\bar Z_{k}$ has the positive parametrization
\begin{equation}
\label{eq:param2}
\bar Z_{k} = \left\{ \tc(k^*,(\fp, h(k^*, \fp)))^{H^T} \circ \fpp^{B^T} \mid \fp \in \RR^{E^0 \setminus \tilde E^0}_{>0}, \, \fpp \in \RR^{n - \tilde s}_{>0} \right\}.
\end{equation}
\end{enumerate}
\end{thm}



Before we prove statements 1 and 2 of Theorem~\ref{main2}, we make two remarks.
\begin{itemize}
\item If the generalized inverse $H \in \RR^{n \times \mathcal{E}}$ of $M^T$ has integer entries, then~\eqref{eq:param} is a rational parametrization. Common generalized inverses such as the Moore-Penrose inverse, however, rarely have this property \cite{Israel2003}. In applications, we construct $H$ by determining the matrix of elementary row operations $P$ that transforms $M^T$ to reduced row echelon form. That is, we find $P \in \RR^{\mathcal{E} \times \mathcal{E}}$ such that $PM^T \in \RR^{\mathcal{E} \times n}$ is the reduced row echelon form of $M^T$. Then we determine $Q \in \{0,1\}^{n \times \mathcal{E}}$ such that $Q P M^T = I$ and hence $\ln x = H \ln \tc$ with $H = Q P \in \RR^{n \times \mathcal{E}}$. That is, we perform Gaussian elimination on~\eqref{loglinear} and then set all free parameters to zero. 
\item As a special case of statement~2, if $\tilde E^0 = E^0$ and equations \eqref{eq:condition} can be solved explicitly for $k^0$ (in terms of $k^*$), that is, if there exists $h \colon \RR^{E^*}_{>0} \to \RR^{E^0}_{>0}$ such that
\[
\kappa (k^*,h(k^*))^C = 1^{\tilde \delta \times 1},
\]
then we obtain the monomial parametrization
\[
\bar{Z}_k = \left\{ \kappa(k^*,h(k^*))^{H^T} \circ \tau^{B^T} \; | \; \tau \in \RR_{>0}^{n - \tilde s} \right\}.
\]
\end{itemize}

\begin{proof}[Proof of statement 1.]
Since $(G,y,\tilde y)$ is weakly reversible, $x \in Z_k$ if and only if $\ln x$ satisfies~\eqref{loglinear}. Now $\im M = \im (\tilde Y I_\mathcal{E}) = \im (\tilde Y I_E) = \tilde S$ and hence $\rank M = \tilde{s}$. Since the kinetic deficiency is zero, we have $\tilde \delta = m - \ell - \tilde s = 0$ and hence $\tilde s = m - \ell$. That is, $M^T$ has full rank $m-\ell$ and hence $\ln \tc \in \im M^T$ for any $\tc \in \RR_{> 0}^{m - \ell}$.
Equivalently, the linear system \eqref{loglinear} has a solution $\ln x$ for any $\tc \in \RR_{>0}^\mathcal{E}$.  
Following Proposition 3 in \cite{MR2014}, we use the generalized inverse $H \in \RR^{n \times \mathcal{E}}$ of $M^T$ and obtain
\[
M^T H \ln \tc = M^T H M^T \ln x = M^T \ln x = \ln \tc .
\]
That is, $\ln x^*=H \ln \tc$  is a solution of~\eqref{loglinear} and hence $x^* = \tc^{H^T} \in Z_k$. In particular, $Z_k \neq \emptyset$.

For any $x \in Z_k$, 
\[
M^T ( \ln(x) - \ln(x^*)) = 0
\]
and, since $\ker M^T = \im M^\perp = \tilde S^\perp$,
\[
\ln(x) - \ln(x^*) \in \tilde S^{\perp}.
\]
We use $B \in \RR^{n \times (n - \tilde s)}$ with $\im B = \tilde{S}^{\perp}$, $\ker B = \{0\}$ and obtain
\[
\ln(x) - \ln(\tilde x^*) = B \alpha
\]
with $\alpha \in \RR^{n-\tilde s}$ and
\[
x = x^* \circ \fpp^{B^T}
\]
with $\fpp = e^\alpha \in \RR^{n-\tilde s}_{>0}$. Equivalently,
\[
Z_k = \left\{ x^* \circ \fpp^{B^T} \mid \fpp \in \RR^{n-\tilde s}_{>0} \right\} 
= \left\{ \tc(k^*,k^0)^{H^T} \circ \fpp^{B^T} \mid \fpp \in \RR^{n-\tilde s}_{> 0} \right\} .
\]
Note that the matrices $M$, $H$, and $B$ do not depend on $k \in \RR_{>0}^{E}$, whereas $\tc = \tc(k) = \tc(k^*,k^0)$. Finally,
\[
\bar Z_{k} = \bigcup_{\fp \in \RR^{E^0}_{>0}} Z_{(k^*,\fp)} 
= \left\{ \tc(k^*,\fp)^{H^T} \circ \fpp^{B^T} \mid \fp \in \RR^{E^0}_{>0}, \, \fpp \in \RR^{n - \tilde s}_{>0} \right\} .
\]
\end{proof}

\begin{proof}[Proof of statement 2.]
If the kinetic deficiency is positive ($\tilde \delta > 0$), then $M^T$ does not have full rank, and \eqref{loglinear} does not have a solution for all right-hand sides. We use $C \in \RR^{\mathcal{E} \times \tilde \delta}$ with $\im C = \ker M$, $\ker C = \{0\}$ and find that~\eqref{loglinear} has a solution if and only if $\ln \tc \in \im M^T = \ker M^\perp = \im C^\perp = \ker C^T$.
Equivalently,
$C^T \ln \tc = 0$,
that is, 
\[
\tc^C = \tc(k^*,k^0)^C = 1^{\tilde \delta \times 1} .
\]
By assumption, 
these $\tilde \delta$ equations can be solved explicitly for $\tilde \delta$ components of $k^0 \in \RR^{E^0}_{>0}$ (in terms of $k^* \in \RR^{E^*}_{>0}$ and the remaining components of $k^0$),
that is,
there exists an explicit function $h \colon \RR^{E^* \cup (E^0 \setminus \tilde E^0)}_{>0} \to \RR^{\tilde E^0}_{>0}$ with $\tilde E^0 \subseteq E^0$, $|\tilde E^0| = \tilde \delta$, and $k^0 = (\tilde k^0, \cdot) \in \RR^{(E^0 \setminus \tilde E^0) \cup \tilde E^0}_{>0}$ such that, for all $k^* \in \RR_{> 0}^{E^*}$ and $\tilde k^0 \in \RR_{> 0}^{E^0 \setminus \tilde E^0}$,
\[
\tc(k^*,(\tilde k^0, h(k^*,\tilde k^0))^{C} = 1^{\tilde \delta \times 1} .
\]
Hence~\eqref{loglinear} has a solution for any $k^* \in \RR^{E^*}_{>0}$ and $\tilde k^0 \in \RR^{E^0 \setminus \tilde E^0}_{>0}$, and from the proof of statement~1 
we obtain the positive parametrization~\eqref{eq:param2}.
\end{proof}

\section{Applications}
\label{sec:applications}

The process of network translation allows to relate a classical CRN to a GCRN with potentially stronger structural properties \cite{J2014}. 
Using this method, we can apply the main results of this paper, Theorems~\ref{main} and~\ref{main2}, to a broad class of mass-action systems studied in the biochemical literature. 

\subsection{Translated chemical reaction networks}
\label{sec:translation}

The following definition was introduced in \cite{J2014} in order to relate a MAS to a dynamically equivalent GMAS.

\begin{dfn}
\label{def:translation}
Let $(G,y)$ with $G=(V,E)$ be a CRN. A GCRN $(G^\tr,y^\tr,\tilde y^\tr)$ with $G^\tr=(V^\tr,E^\tr)$ if a \emph{translation} of $(G,y)$ is there exists a map $g \colon E \to E^\tr$ such that $g(i \to j) = i^\tr \to j^\tr$ with $i \to j \in E$ and $i^\tr \to j^\tr \in E^\tr$ implies (i) $y^\tr(j^\tr)-y^\tr(i^\tr) = y(j) - y(i)$ and (ii) $\tilde y^\tr(i^\tr) = y(i)$.
\end{dfn}

In other words, a GCRN is a translation of a given CRN if there is a map between reactions of the two networks which (i) preserves reaction vectors and (ii) relates source complexes in the CRN to kinetic complexes in the GCRN. Definition~\ref{def:translation} is more general than Definition 6 in \cite{J2014}. In that work, GCRNs were defined as in \cite{MR2012} which required $y^\tr$ and $\tilde y^\tr$ to be injective. Here, GCRNs are defined as in \cite{MR2014} which allows $y^\tr$ and $\tilde y^\tr$ to be noninjective. 


\begin{lem}
\label{lemma:equiv1}
Let $(G,y)$ be a CRN, and let $k \in \RR^{E}_{> 0}$ be a rate vector. Further, let the GCRN $(G^\tr, y^\tr, \tilde{y}^\tr)$ be a translation of $(G,y)$, and let $k^\tr \in \RR_{> 0}^{E^\tr}$ be a rate vector with $k^\tr_{i^\tr\to j^\tr} =  k_{i \to j}$ if $g(i \to j) = i^\tr \to j^\tr$. Then the MAS $(G_k,y)$ and the GMAS $(G^\tr_{k^\tr}, y^\tr, \tilde y^\tr)$ are dynamically equivalent, that is, the associated ODEs agree, cf.~\eqref{mas} and~\eqref{gmas}. 
\end{lem}

\begin{proof}
The ODEs associated with the MAS $(G_k,y)$ are determined by $f^G_k(x)$, cf.~\eqref{mas},
whereas the ODEs associated with the GMAS $(G^\tr_{k^\tr},y^\tr,\tilde y^\tr)$ are determined by $f^{G^\tr}_{k^\tr}(x)$, cf.~\eqref{gmas}.
By Definition~\ref{def:translation} and the construction of $k^\tr$, we have
\[ f^G_{k}(x) = \sum_{i \to j \in E} k_{i \to j} \, x^{y(i)} \, (y(j) - y(i))= \sum_{i^\tr \to j^\tr \in E^\tr} k^\tr_{i^\tr \to j^\tr} \, x^{\tilde y^\tr(i^\tr)} \, (y^\tr(j^\tr) - y^\tr(i^\tr)) = f^{G^\tr}_{k^\tr}(x).\]
\end{proof}

Lemmas~\ref{lemma:equiv1} and~\ref{lemma:equiv2} provide a framework for parametrizing the set of positive equilibria of a (classical) MAS~\eqref{mas}, by applying Theorems~\ref{main} and~\ref{main2}.

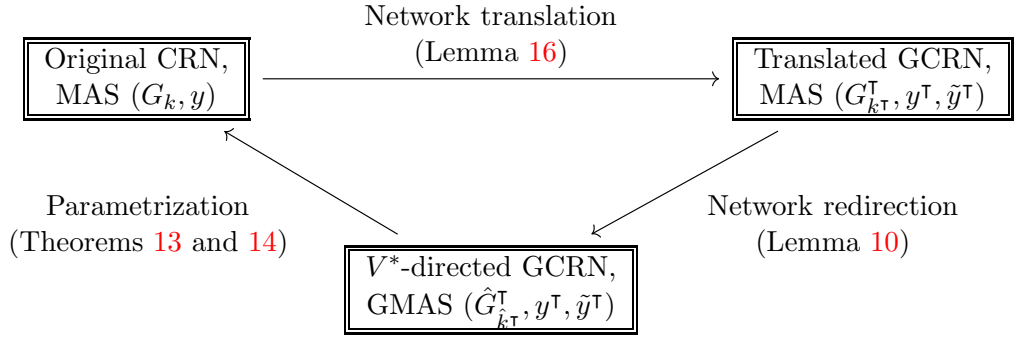
\[
\begin{tikzcd}
\mbox{\doublebox{\begin{tabular}{c} Original CRN, \\ MAS $(G_k,y)$ \end{tabular}}} 
\arrow[rr,"\begin{tabular}{c} Network translation \\ (Lemma~\ref{lemma:equiv1}) \end{tabular}"] & & \mbox{\doublebox{\begin{tabular}{c} Translated GCRN, \\ MAS $(G_{k^\tr}^\tr,y^\tr, \tilde y^\tr)$ \end{tabular}}}
\arrow[ddl,"\begin{tabular}{c} Network redirection \\ (Lemma~\ref{lemma:equiv2}) \end{tabular}"] \\
& & \\
& \mbox{\doublebox{\begin{tabular}{c} \vd GCRN, \\ GMAS $(\rd G^\tr_{\rd k^\tr},y^\tr,\tilde y^\tr)$ \end{tabular}}} 
\arrow[luu,"\begin{tabular}{c} Parametrization \\ (Theorems \ref{main} and~\ref{main2}) \end{tabular}"] & 
\end{tikzcd}
\]

In biochemical applications, a suitable GCRN that corresponds to a given CRN may not be apparent. In particular, in order to apply Theorem~\ref{main}, we want the translated network to have effective deficiency zero, and to apply Theorem~\ref{main2}, we want the kinetic deficiency to be as low as possible and the translated and \vd network to be weakly reversible.

A \emph{translation scheme} involves the addition of linear combinations of species to each side of a reaction arrow~\cite{J2014}. This operation preserves reaction vectors and establishes a correspondence between source complexes in the original network and kinetic complexes in the new one. For small networks, this may suffice to create a suitably well-connected translation; however, it is extremely challenging for large networks. Computational approaches to optimal network translation have been conducted in \cite{J2015} and \cite{Tonello2017}.

\subsection{Examples}
\label{sec:examples}

The following examples are drawn from the biochemical literature.

\begin{exa}
Recall the histidine kinase network~\eqref{intro-example} from the introduction and apply the following translation scheme:
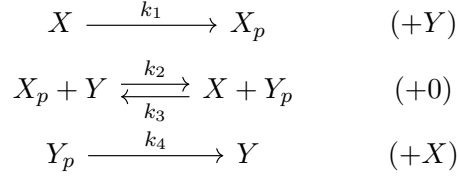
\begin{equation*}
\begin{tikzcd}
X \arrow[r, "k_1"] & X_p & (+ Y)\\[-0.2in]
X_p + Y \arrow[r,  yshift=+0.5ex,"k_2"] & X + Y_p \arrow[l,  yshift=-0.5ex,"k_3"] & (+ 0) \\[-0.2in]  
Y_p \arrow[r, "k_4"] & Y & (+ X)
\end{tikzcd}
\end{equation*}
The resulting GCRN together with an additional phantom edge yields a weakly reversible GCRN, given by the (edge labeled) graph

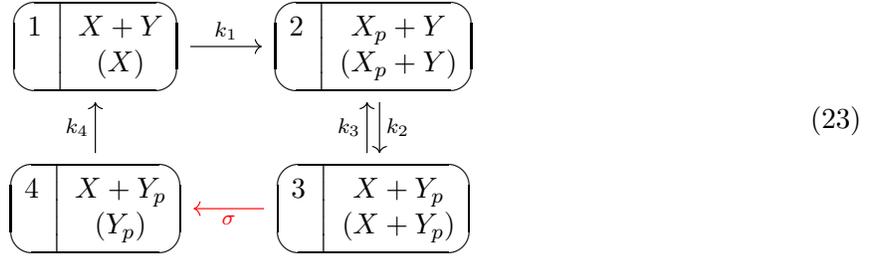
\begin{equation}
\label{example233}
\begin{tikzcd}
\ovalbox{$\begin{array}{c} 1 \\ \\ \end{array} \Bigg\lvert \begin{array}{c} X + Y \\  (X) \end{array}$} \arrow[r,"k_1"] & 
\ovalbox{$\begin{array}{c} 2 \\ \\ \end{array}\Bigg\lvert \begin{array}{c} X_p + Y \\  (X_p+Y) \end{array}$} \arrow[d,xshift=+0.5ex,"k_2"]\\
\ovalbox{$\begin{array}{c} 4 \\ \\ \end{array}\Bigg\lvert \begin{array}{c}  X + Y_p \\ (Y_p) \end{array}$} \arrow[u,"k_4"] & 
\ovalbox{$\begin{array}{c} 3 \\ \\ \end{array}\Bigg\lvert \begin{array}{c} X + Y_p \\ (X+Y_p) \end{array}$} \arrow[u,xshift=-0.5ex,"k_3"] \arrow[l,color=red,"\fp"]
\end{tikzcd}
\end{equation}
The stoichiometric complex $X+Y_p$ appears twice in~\eqref{example233}, 
specifically, $[3]=[4]=\{3,4\}$, and the network is \vd for $V^* = \{1,2,3\}$.
The network has a stoichiometric deficiency of one ($\delta = 1$) and a kinetic deficiency of zero ($\tilde \delta = 0)$.  
The condensed network is given by the following graph:
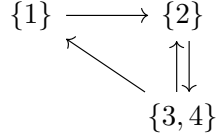
\[
\begin{tikzcd}
\{1\} \arrow[r] & \{2\} \arrow[d,xshift=0.5ex] \\
& \{3,4\} \arrow[u,xshift=-0.5ex] \arrow[lu]
\end{tikzcd}
\]
It has a deficiency of zero ($\delta' = 0$). Theorem~\ref{main} guarantees that the equilibrium set coincides with the parametrized set of complex-balanced equilibria. Furthermore, since $\tilde \delta = 0$ and~\eqref{example233} is weakly reversible, Theorem~\ref{main2} guarantees that there is a positive parametrization of the form~\eqref{eq:param}. 


By the construction preceding Theorem~\ref{main2}, we compute the matrix $M$ (and further~$H$ and~$B$). In particular, we choose a spanning forest $F = (V, \mathcal{E})$ for the graph~\eqref{example233} with edges $1 \to 2$, $1 \to 3$, and $1 \to 4$, and we compute the corresponding differences of kinetic complexes $X_p + Y - X$, $X + Y_p - X$, and $Y_p - X$:
\[
M = \begin{array}{c} X \\ X_p \\ Y \\ Y_p \end{array} \hspace{-1ex}
\left[ \begin{array}{ccc} -1 & 0 & -1 \\ 1 & 0 & 0 \\ 1 & 0 & 0 \\ 0 & 1 & 1\end{array} \right], \hspace{0.25in} H = \left[ \begin{array}{ccc} 0 & 1 & -1 \\ 1 & 1 & -1 \\ 0 & 0 & 0 \\ 0 & 1 & 0 \end{array} \right], \hspace{0.25in} B = \left[ \begin{array}{c} 0 \\ -1 \\ 1 \\ 0 \end{array} \right] .
\]
Thereby $M^T H M^T = M^T$, that is, $H$ is a generalized inverse of $M^T$,
and $\im B = \ker M^T$.

In order to determine the parametrization~\eqref{eq:param}, it remains to compute the tree constants $K=K(k^*,\fp)$ of the graph~\eqref{example233} and their quotients $\tc=\tc(k^*,\fp)$. 
We find
\[K_1 = k_2k_4\fp, \; \; \; K_2 = k_1(k_3+\fp)k_4, \; \; \; K_3 = k_1k_2k_4, \; \mbox{ and } \; K_4 = k_1k_2\fp.\]
Taking the spanning forest $F = (V, \mathcal{E})$ as above gives
\[
\tc_1 = \frac{K_2}{K_1} = \frac{k_1(k_3 + \fp)}{k_2\fp}, \hspace{0.25in} \tc_2 = \frac{K_3}{K_1} = \frac{k_1}{\fp}, \hspace{0.25in}
\tc_3  = \frac{K_4}{K_1}= \frac{k_1}{k_4}.\]
As a consequence, the rational parametrization~\eqref{eq:param} amounts to
\[
\left \{ \quad
\begin{aligned}
x &= \tc_2^1 \tc_3^{-1} \cdot 1 &&= \frac{k_4}{\fp} , \\ 
x_p &= \tc_1^1 \tc_2^1 \tc_3^{-1} \cdot \fpp^{-1} &&= \frac{k_1(k_3+\fp)k_4}{k_2\fp^2\fpp} , \\ 
y &= 1 \cdot \fpp^1 &&= \fpp , \\ 
y_p &= \tc_2^1 \cdot 1 &&= \frac{k_1}{\fp} ,
\end{aligned}
\right.
\]
where $\fp,\fpp > 0$.
\end{exa}
\begin{exa}
Consider the following EnvZ-OmpR signaling pathway, which was first proposed in \cite{Sh-F}, together with the translation scheme proposed in \cite{J2014}:
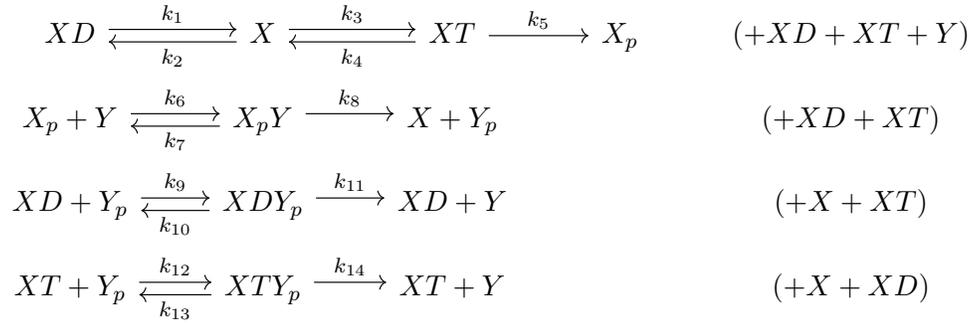
\[
      \begin{tikzcd}
XD \arrow[r,yshift=+0.5ex,"k_1"] & X \arrow[l,yshift=-0.5ex,"k_2"] \arrow[r,yshift=+0.5ex,"k_3"]& XT \arrow[l,yshift=-0.5ex,"k_4"] \arrow[r,"k_5"]& X_p & (+ XD + XT + Y)\\[-0.1in]
X_p + Y \arrow[r,yshift=+0.5ex,"k_6"] & X_pY \arrow[l,yshift=-0.5ex,"k_7"] \arrow[r,yshift=+0.5ex,"k_8"]& X + Y_p &  & (+ XD + XT) \\[-0.1in]
XD + Y_p \arrow[r,yshift=+0.5ex,"k_9"] & XDY_p \arrow[l,yshift=-0.5ex,"k_{10}"] \arrow[r,yshift=+0.5ex,"k_{11}"]& XD + Y &  & (+ X + XT) \\[-0.1in]
XT + Y_p \arrow[r,yshift=+0.5ex,"k_{12}"] & XTY_p \arrow[l,yshift=-0.5ex,"k_{13}"] \arrow[r,yshift=+0.5ex,"k_{14}"]& XT + Y &  & (+ X + XD) \\[-0.1in]
      \end{tikzcd}
\]
The resulting GCRN together with an additional phantom edge yields a weakly reversible GCRN, given by the (edge labeled) graph
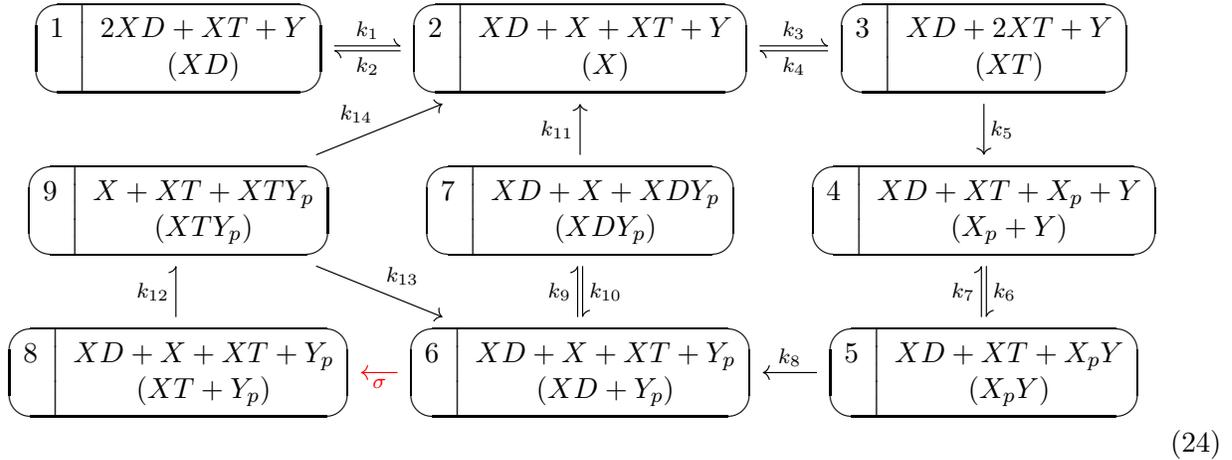
\begin{equation}\label{envz_ompr_directed}
        \begin{tikzcd}[column sep=0.5cm]
            \mbox{\ovalbox{$\begin{array}{c} 1 \\ \\ \end{array}\Bigg\lvert \begin{array}{c} 2XD + XT + Y \\ (XD) \end{array}$}} \arrow[r, rightharpoonup, yshift=+0.2ex, "k_{1}"]
          & \mbox{\ovalbox{$\begin{array}{c}  2 \\ \\ \end{array}\Bigg\lvert \begin{array}{c}XD + X + XT + Y \\ (X) \end{array}$}} \arrow[l, rightharpoonup, yshift=-0.2ex, "k_{2}"] \arrow[r, rightharpoonup, yshift=+0.2ex, "k_{3}"]
          & \mbox{\ovalbox{$\begin{array}{c}  3 \\ \\ \end{array}\Bigg\lvert \begin{array}{c} XD + 2XT + Y \\ (XT) \end{array}$}} \arrow[l, rightharpoonup, yshift=-0.2ex, "k_{4}"] \arrow[d, "k_{5}"] & \\
            \mbox{\ovalbox{$\begin{array}{c} 9 \\ \\ \end{array}\Bigg\lvert \begin{array}{c} X + XT + XTY_p \\ (XTY_p) \end{array}$}} \arrow[ur, "k_{14}"] \arrow[rd, "k_{13}"]
          & \mbox{\ovalbox{$\begin{array}{c} 7 \\ \\ \end{array}\Bigg\lvert \begin{array}{c} XD + X + XDY_p \\ (XDY_p) \end{array}$}} \arrow[u, "k_{11}"] \arrow[d, rightharpoonup, "k_{10}", xshift=+0.2ex]
          & \mbox{\ovalbox{$\begin{array}{c}  4 \\ \\ \end{array}\Bigg\lvert \begin{array}{c} XD + XT + X_p + Y \\ (X_p + Y) \end{array}$}} \arrow[d, rightharpoonup, "k_{6}", xshift=+0.2ex] \\
          \mbox{\ovalbox{$\begin{array}{c} 8 \\ \\ \end{array}\Bigg\lvert \begin{array}{c} XD + X + XT + Y_p \\ (XT + Y_p) \end{array}$}} \arrow[u, rightharpoonup, "k_{12}", xshift=-0.2ex] 
          &  \mbox{\ovalbox{$\begin{array}{c} 6 \\ \\ \end{array}\Bigg\lvert \begin{array}{c} XD + X + XT + Y_p \\ (XD + Y_p) \end{array}$}} \arrow[u, rightharpoonup, "k_{9}", xshift=-0.2ex] \arrow[l,color=red, "\fp"]
          & \mbox{\ovalbox{$\begin{array}{c} 5 \\ \\ \end{array}\Bigg\lvert \begin{array}{c} XD + XT + X_pY \\ (X_pY) \end{array}$}} \arrow[l, "k_8"'] \arrow[u, rightharpoonup, "k_{7}", xshift=-0.2ex] &
      \end{tikzcd}
    \end{equation}
Thereby $6 \to 8$ is the phantom edge (with label $\fp>0$) since $y(6) = y(8) = XD + X + XT + Y_p$. 
The network is \vd for $V^* = V\setminus\{8\}$.
It can be quickly checked that the condensed graph $G'$ has deficiency zero, so that~\eqref{envz_ompr_directed} has an effective deficiency of zero ($\delta' = 0$). It follows from Theorem~\ref{main} that every equilibrium point is in the parametrized set of CBE (i.e.~$X_k = \bar{Z}_k$). It can also be checked that~\eqref{envz_ompr_directed} has a kinetic deficiency of one ($\tilde \delta = 1$). Hence, in order to apply Theorem~\ref{main2} (statement 2), we need to first determine if there is $\fp=h(k^*)$ such that $\tc(k^*,h(k^*))^{C} = 1$.


We choose the spanning forest $F = (V, \mathcal{E})$ for the graph~\eqref{envz_ompr_directed} consisting of the edges $1 \to i$ for $i=2, \ldots, 9$. We compute the following matrices:
\[\tiny 
M = \hspace{-1ex} \begin{array}{c} XD \\ X \\ XT \\ X_p \\ Y \\ X_pY \\ Y_p \\ XDY_p \\ XTY_p \end{array} \hspace{-1ex} \left[ \begin{array}{cccccccc} \hspace{-0.05in}  -1 \hspace{-0.05in} & \hspace{-0.05in}  -1  \hspace{-0.05in} &  \hspace{-0.05in}  -1  \hspace{-0.05in}&  \hspace{-0.05in}  -1  \hspace{-0.05in}&  0 &  \hspace{-0.05in}  -1 \hspace{-0.05in}&  \hspace{-0.05in}  -1 \hspace{-0.05in} &  \hspace{-0.05in}  -1 \hspace{-0.05in}\\1 & 0 & 0 & 0 & 0 & 0 & 0 & 0\\0 & 1 & 0 & 0 & 0 & 0 & 1 & 0\\0 & 0 & 1 & 0 & 0 & 0 & 0 & 0\\0 & 0 & 1 & 0 & 0 & 0 & 0 & 0\\0 & 0 & 0 & 1 & 0 & 0 & 0 & 0\\0 & 0 & 0 & 0 & 1 & 0 & 1 & 0\\0 & 0 & 0 & 0 & 0 & 1 & 0 & 0\\0 & 0 & 0 & 0 & 0 & 0 & 0 & 1 \end{array} \right], \; H =  \left[ \begin{array}{cccccccc} 0&0&0&0&0&0&0& \hspace{-0.05in}-1 \hspace{-0.05in}\\1&0&0&0&0&0&0& \hspace{-0.05in}-1 \hspace{-0.05in}\\0&1&0&0&0&0&0& \hspace{-0.05in}-1 \hspace{-0.05in}\\0&0&1&0&0&0&0& \hspace{-0.05in}-1 \hspace{-0.05in}\\0&0&0&0&0&0&0&0\\0&0&0&1&0&0&0& \hspace{-0.05in}-1 \hspace{-0.05in}\\0&0&0&0&1&0&0&0\\0&0&0&0&0&1&0& \hspace{-0.05in}-1 \hspace{-0.05in}\\0&0&0&0&0&0&0&0 \end{array} \right], \; B = \left[ \begin{array}{cc}  0 & 1 \\ 0 & 1 \\ 0 & 1 \\  \hspace{-0.05in}-1 \hspace{-0.05in} & 1 \\ 1 & 0 \\ 0 & 1 \\ 0 & 0 \\ 0 & 1 \\ 0 & 1 \end{array} \right], \; C^T = \left[ \begin{array}{c}  0  \\  \hspace{-0.05in}-1 \hspace{-0.05in} \\ 0 \\ 0 \\  \hspace{-0.05in}-1 \hspace{-0.05in} \\ 0 \\ 1 \\ 0 \end{array} \right] 
\]

And we find the following tree constants:
\[ \small \begin{split}
K_{1} &= (k_4 + k_5) (((k_9 + \fp) k_{14} + k_9 k_{13}) k_{11} + \fp k_{14} k_{10}) k_2 k_6 k_8 k_{12}\\
K_{2} &= (k_{4} + k_{5}) (((k_{9} + \fp) k_{14} + k_{9} k_{13}) k_{11} + \fp k_{14} k_{10}) k_{1} k_{6} k_{8} k_{12}\\
K_{3} &= k_{6} (((k_{9} + \fp) k_{14} + k_{9} k_{13}) k_{11} + \fp k_{14} k_{10}) k_{12} k_{1} k_{8} k_{3}\\
K_{4} &= (k_{7} + k_{8}) (((k_{9} + \fp) k_{14} + k_{9} k_{13}) k_{11} + \fp k_{14} k_{10}) k_{5} k_{1} k_{3} k_{12}\\
K_{5} &= k_{5} (((k_{9} + \fp) k_{14} + k_{9} k_{13}) k_{11} + \fp k_{14} k_{10}) k_{12} k_{1} k_{6} k_{3}\\
K_{6} &= (k_{10} + k_{11}) k_{12} (k_{13} + k_{14}) k_{1} k_{3} k_{5} k_{6} k_{8}\\
K_{7} &= k_{1} k_{12} k_{3} k_{5} k_{6} k_{8} k_{9} (k_{13} + k_{14})\\
K_{8} &= (k_{13} + k_{14}) k_{1} k_{3} k_{5} k_{6} k_{8} \fp (k_{10} + k_{11})\\
K_{9} &= k_{1} k_{3} k_{5} k_{6} k_{8} \fp (k_{10} + k_{11}) k_{12}
\end{split} \]
Constructing $\tc=\tc(k^*,\fp)$ according to the spanning forest $F = (V,\mathcal{E})$ as above gives the $\tilde \delta = 1$ condition
\[
\tc(k^*,\fp)^C = \left(\frac{K_{3}}{K_{1}}\right)^{-1}\left(\frac{K_{6}}{K_{1}}\right)^{-1} \left(\frac{K_{8}}{K_{1}}\right) = \frac{k_2(k_4+k_5)\fp}{k_1k_3k_{12}} = 1,
\]
which can be solved explicitly for $\fp$ (in terms of $k^*$),
\[
\fp = \frac{k_1k_3k_{12}}{k_2(k_4+k_5)}.
\]
By Theorem~\ref{main2} (statement 2), we have a monomial parametrization of the form~\eqref{eq:param2}. In particular, we obtain:
\[ \small \begin{split}
XD & = \left( \frac{((k_{2} (k_{4}+k_{5}) (k_{13}+k_{14}) k_{9}+k_{1} k_{12} k_{14} k_{3}) k_{11}+k_{1} k_{3} k_{12} k_{14} k_{10}) (k_{4}+k_{5}) k_{2}}{(k_{10}+k_{11}) k_{12} k_{5} k_{3}^2 k_{1}^2 } \right)\fpp_1\\
X & = \left( \frac{(((k_{2} (k_{4}+k_{5}) k_{9}+k_{1} k_{3} k_{12} k_{14}+k_{13} k_{2} k_{9} (k_{4}+k_{5})) k_{11}+k_{1} k_{3} k_{12} k_{14} k_{10}) (k_{4}+k_{5})}{ (k_{10}+k_{11}) k_{12} k_{1} k_{5} k_{3}^2} \right)\fpp_1\\
XT & = \left( \frac{((k_{2} (k_{4}+k_{5}) k_{9}+k_{1} k_{3} k_{12}) k_{14}+k_{13} k_{2} k_{9} (k_{4}+k_{5})) k_{11}+k_{1} k_{3} k_{12} k_{14} k_{10}}{(k_{10}+k_{11}) k_{12} k_{1} k_{3} k_{5} } \right)\fpp_1\\
X_p & = \left( \frac{(((k_{2} (k_{4}+k_{5}) k_{9}+k_{1} k_{3} k_{12}) k_{14}+k_{13} k_{2} k_{9} (k_{4}+k_{5})) k_{11}+k_{1} k_{3} k_{12} k_{14} k_{10}) (k_{7}+k_{8}) }{(k_{10}+k_{11}) k_{12} k_{3} k_{1} k_{8} k_{6}} \right)\frac{\fpp_1
}{\fpp_2}\\
Y & = \fpp_2\\
X_pY & = \left( \frac{((k_{2} (k_{4}+k_{5}) k_{9}+k_{1} k_{3} k_{12}) k_{14}+k_{13} k_{2} k_{9} (k_{4}+k_{5})) k_{11}+k_{1} k_{3} k_{12} k_{14} k_{10}}{(k_{10}+k_{11}) k_{12} k_{3} k_{1} k_{8}} \right)\fpp_1\\
Y_p & = \left( \frac{((k_{2} (k_{4}+k_{5}) k_{9}+k_{1} k_{3} k_{12}) k_{14}+k_{13} k_{2} k_{9} (k_{4}+k_{5})) k_{11}+k_{1} k_{3} k_{12} k_{14} k_{10}}{k_{5} k_{3} k_{1} (k_{13}+k_{14}) (k_{10}+k_{11})} \right)\\
XDY_p & = \left( \frac{k_{2} (k_{4}+k_{5}) (k_{13}+k_{14}) k_{9}}{(k_{10}+k_{11}) k_{1} k_{3} k_{12}} \right) \fpp_1\\
XTY_p & = \fpp_1
\end{split}\]
over $\fpp_1, \fpp_2 \in \RR_{>0}$. This parametrization was obtained via alternative methods in \cite{M-D-S-C} and \cite{J2014}. 

Note that the concentration of $Y_p$ does not depend upon either parameter $\fpp_1$ or $\fpp_2$. Hence it takes the same value at every positive steady state. This property has been called \emph{absolute concentration robustness} (ACR) in the literature, and the robust steady state value of $Y_p$ has been obtained by other methods in \cite{Sh-F,Karp,Tonello2017,M-D-S-C}.
\end{exa}

\begin{exa}
Consider the model for the Shuttled WNT signaling pathway from \cite{G-H-R-S}, which has a deficiency of four ($\delta = 4$), taken with the following translation scheme:
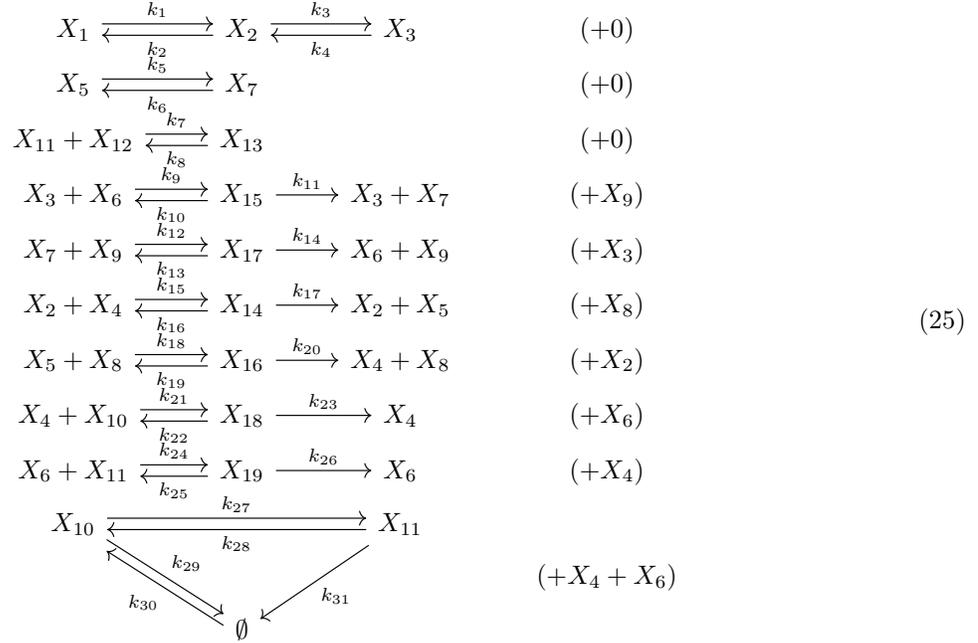
\begin{equation}\small
\label{wnt}
      \begin{tikzcd}
X_1 \arrow[r, yshift=+0.5ex, "k_1"] & X_2 \arrow[l, yshift=-0.5ex, "k_2"] \arrow[r, yshift=0.5ex, "k_{3}"] & X_3 \arrow[l,yshift=-0.5ex, "k_{4}"] & (+ 0)\\[-0.2in]
X_5 \arrow[r,  yshift=+0.5ex,"k_{5}"] & X_7 \arrow[l,  yshift=-0.5ex,"k_{6}"] & & (+ 0) \\[-0.2in]
X_{11}+ X_{12} \arrow[r,  yshift=+0.5ex,"k_{7}"] & X_{13} \arrow[l,  yshift=-0.5ex,"k_{8}"] & & (+ 0)  \\[-0.2in]  
X_{3} + X_6 \arrow[r,  yshift=+0.5ex,"k_{9}"] & X_{15} \arrow[l,  yshift=-0.5ex,"k_{10}"] \arrow[r, "k_{11}"] & X_3 + X_7 & (+ X_9)\\[-0.2in] 
X_{7} + X_9 \arrow[r,  yshift=+0.5ex,"k_{12}"] & X_{17} \arrow[l,  yshift=-0.5ex,"k_{13}"] \arrow[r, "k_{14}"] & X_6 + X_9 & (+ X_3)\\[-0.2in]  
X_{2} + X_4 \arrow[r,  yshift=+0.5ex,"k_{15}"] & X_{14} \arrow[l,  yshift=-0.5ex,"k_{16}"] \arrow[r, "k_{17}"] & X_2 + X_5 & (+ X_8)\\[-0.2in]  
X_{5} + X_8 \arrow[r,  yshift=+0.5ex,"k_{18}"] & X_{16} \arrow[l,  yshift=-0.5ex,"k_{19}"] \arrow[r, "k_{20}"] & X_4 + X_8 & (+ X_2)\\[-0.2in]  
X_{4} + X_{10} \arrow[r,  yshift=+0.5ex,"k_{21}"] & X_{18} \arrow[l,  yshift=-0.5ex,"k_{22}"] \arrow[r, "k_{23}"] & X_4 & (+ X_6)\\[-0.2in] 
X_{6} + X_{11} \arrow[r,  yshift=+0.5ex,"k_{24}"] & X_{19} \arrow[l,  yshift=-0.5ex,"k_{25}"] \arrow[r, "k_{26}"] & X_6 & (+ X_4) \\[-0.2in]  
X_{10} \arrow[rr,  yshift=+0.5ex,"k_{27}"] \arrow[ddr, yshift=0.5ex,"k_{29}"] & & X_{11} \arrow[ll,  yshift=-0.5ex,"k_{28}"] \arrow[ddl,"k_{31}"] & \\[-0.2in]
& &  & (+ X_4 + X_6)\\[-0.2in]
& \emptyset \arrow[uul,yshift=-0.5ex,"k_{30}"] &  &
      \end{tikzcd}
\end{equation}
In the representation above, we have kept the indexing of the species $X_1$ through $X_{19}$ as in \cite{G-H-R-S}, but renamed the rate constants. 
Via Lemmas~\ref{lemma:equiv1} and~\ref{lemma:equiv2},
the network corresponds to a weakly reversible, \vd GCRN:
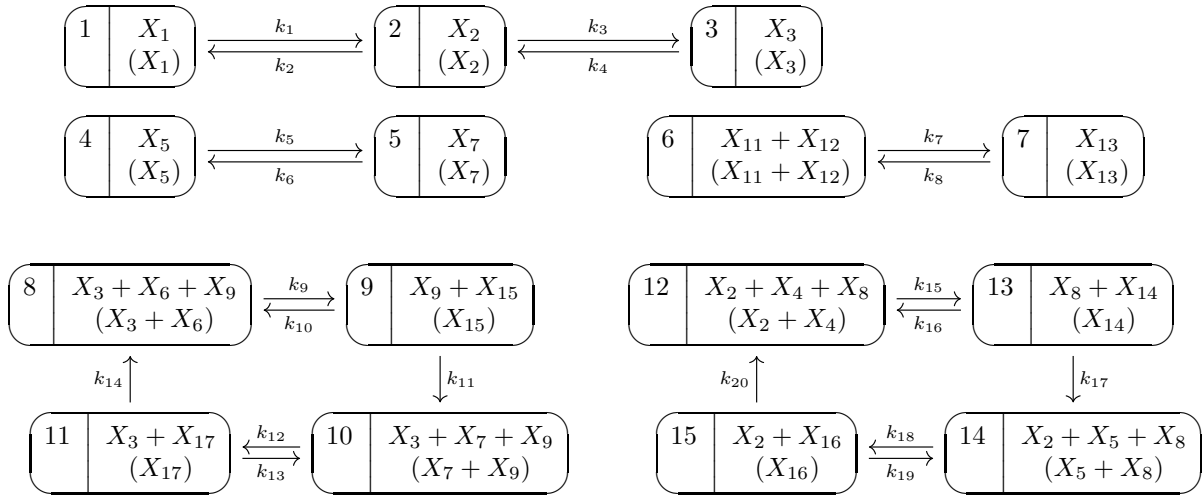
\[
\small
\begin{tikzcd}[column sep=0.5cm]
\mbox{\ovalbox{$\begin{array}{c}  1 \\ \\ \end{array}\Bigg\lvert \begin{array}{c} X_1 \\ (X_1) \end{array}$}} \arrow[r, yshift=+0.5ex, "k_1"] & \mbox{\ovalbox{$\begin{array}{c} 2 \\ \\ \end{array}\Bigg\lvert \begin{array}{c} X_2 \\ (X_2) \end{array}$}} \arrow[l, yshift=-0.5ex, "k_2"] \arrow[r, yshift=0.5ex, "k_{3}"] & \mbox{\ovalbox{$\begin{array}{c} 3 \\ \\ \end{array}\Bigg\lvert \begin{array}{c} X_3 \\ (X_3) \end{array}$}} \arrow[l,yshift=-0.5ex, "k_{4}"] &\\[-0.2in]
\mbox{\ovalbox{$\begin{array}{c} 4 \\ \\ \end{array}\Bigg\lvert \begin{array}{c} X_5 \\ (X_5) \end{array}$}} \arrow[r,  yshift=+0.5ex,"k_{5}"] & \mbox{\ovalbox{$\begin{array}{c} 5 \\ \\ \end{array}\Bigg\lvert \begin{array}{c} X_7 \\ (X_7) \end{array}$}} \arrow[l,  yshift=-0.5ex,"k_{6}"] & \mbox{\ovalbox{$\begin{array}{c} 6 \\ \\ \end{array}\Bigg\lvert \begin{array}{c} X_{11}+ X_{12} \\ (X_{11} + X_{12}) \end{array}$}} \arrow[r,  yshift=+0.5ex,"k_{7}"] & \mbox{\ovalbox{$\begin{array}{c} 7 \\ \\ \end{array}\Bigg\lvert \begin{array}{c} X_{13} \\ (X_{13}) \end{array}$}} \arrow[l,  yshift=-0.5ex,"k_{8}"] \\
\mbox{\ovalbox{$\begin{array}{c}  8 \\ \\ \end{array}\Bigg\lvert \begin{array}{c} X_3+X_6+X_{9} \\ (X_3 + X_{6}) \end{array}$}}  \arrow[r,yshift=0.5ex, "k_{9}"]  & \mbox{\ovalbox{$\begin{array}{c}  9 \\ \\ \end{array}\Bigg\lvert \begin{array}{c} X_9 + X_{15} \\ (X_{15}) \end{array}$}} \arrow[l,yshift=-0.5ex,"k_{10}"] \arrow[d,"k_{11}"] & \mbox{\ovalbox{$\begin{array}{c}  12 \\ \\ \end{array}\Bigg\lvert \begin{array}{c} X_2+X_4+X_{8} \\ (X_2 + X_{4}) \end{array}$}}  \arrow[r,yshift=0.5ex, "k_{15}"]  & \mbox{\ovalbox{$\begin{array}{c}  13 \\ \\ \end{array}\Bigg\lvert \begin{array}{c} X_8 + X_{14} \\ (X_{14}) \end{array}$}} \arrow[l,yshift=-0.5ex,"k_{16}"] \arrow[d,"k_{17}"]\\
\mbox{\ovalbox{$\begin{array}{c} 11 \\ \\ \end{array}\Bigg\lvert \begin{array}{c} X_3+X_{17} \\ (X_{17}) \end{array}$}} \arrow[r,yshift=-0.5ex,"k_{13}"'] \arrow[u,"k_{14}"] &
\mbox{\ovalbox{$\begin{array}{c}  10 \\ \\ \end{array}\Bigg\lvert \begin{array}{c} X_3+X_7 + X_9 \\ (X_7 + X_9) \end{array}$}} \arrow[l,yshift=0.5ex,"k_{12}"']& \mbox{\ovalbox{$\begin{array}{c}  15 \\ \\ \end{array}\Bigg\lvert \begin{array}{c} X_2+X_{16} \\ (X_{16}) \end{array}$}} \arrow[r,yshift=-0.5ex,"k_{19}"'] \arrow[u,"k_{20}"] &
\mbox{\ovalbox{$\begin{array}{c}  14 \\ \\ \end{array}\Bigg\lvert \begin{array}{c} X_2+X_5 + X_8 \\ (X_5 + X_8) \end{array}$}} \arrow[l,yshift=0.5ex,"k_{18}"']
      \end{tikzcd}
\]
\begin{equation}\label{mess1}
\begin{tikzcd}[column sep=0.5cm]
\mbox{\ovalbox{$\begin{array}{c} 16 \\ \\ \end{array}\Bigg\lvert \begin{array}{c} X_4+X_6+X_{10} \\ (X_4 + X_{10}) \end{array}$}} \arrow[d,"k_{21}"] & \mbox{\ovalbox{$\begin{array}{c}  17 \\ \\ \end{array}\Bigg\lvert \begin{array}{c} X_4 + X_6 + X_{10} \\ (X_{10}) \end{array}$}} \arrow[l,color=red, "\fp_1"'] \arrow[r,  yshift=+0.5ex, "k_{27}"] \arrow[d,xshift=0.5ex,"k_{29}"]& \mbox{\ovalbox{$\begin{array}{c}  18 \\ \\ \end{array}\Bigg\lvert \begin{array}{c} X_4+X_6+X_{11} \\ (X_{11}) \end{array}$}} \arrow[r,color=red, "\fp_{2}"] \arrow[l, yshift=-0.5ex, "k_{28}"] \arrow[dl,"k_{31}"]& \mbox{\ovalbox{$\begin{array}{c} 19 \\ \\ \end{array}\Bigg\lvert \begin{array}{c} X_4+X_6+X_{11} \\ (X_6+X_{11}) \end{array}$}} \arrow[dl,"k_{24}"]\\
\mbox{\ovalbox{$\begin{array}{c}  20 \\ \\ \end{array}\Bigg\lvert \begin{array}{c} X_6+X_{18} \\ (X_{18}) \end{array}$}} \arrow[r,"k_{23}"'] \arrow[ur,"k_{22}"'] &
\mbox{\ovalbox{$\begin{array}{c} 21 \\ \\ \end{array}\Bigg\lvert \begin{array}{c} X_4+X_6 \\ (0) \end{array}$}} \arrow[u,xshift=-0.5ex,"k_{30}"] & \mbox{\ovalbox{$\begin{array}{c}  22 \\ \\ \end{array}\Bigg\lvert \begin{array}{c} X_4 + X_{19} \\ (X_{19}) \end{array}$}} \arrow[l,"k_{26}"] \arrow[u,"k_{25}"'] &  
      \end{tikzcd}
\end{equation}

\noindent 
Thereby, $17 \to 16$ and $18 \to 19$ (with labels $\fp_1> 0$ and $\fp_2 > 0$) are phantom egdes 
since $y(16)=y(17)=X_4+X_6+X_{10}$ and $y(18)=y(19)=X_4+X_6+X_{11}$. 
The network is \vd for $V^* = V \setminus \{16,19\}$.
It can be quickly checked that the GCRN has a stoichiometric deficiency of two ($\delta = 2$) but effective and kinetic deficiencies of zero ($\delta' = 0$ and $\tilde \delta = 0$). By Theorems~\ref{main} and~\ref{main2} (statement 1), the equilibrium set can be parametrized by~\eqref{eq:param}.

Explicitly, we choose the spanning forest $F = (V, \mathcal{E})$ for the graph~\eqref{mess1} consisting of the edges $1 \to i$ for $i \in \{2, 3 \}$, $4 \to 5$, $6 \to 7$, $8 \to i$ for $i \in \{ 9, 10, 11\}$, $12 \to i$ for $i \in \{ 13, 14, 15 \}$, and $16 \to i$ for $i \in \{ 17, \ldots, 22 \}$. Then we compute the corresponding matrix $M$:
\[
\tiny
M = \left[ \begin{array}{cccccccccccccccc}
-1&-1&0&0&0&0&0&0&0&0&0&0&0&0&0&0\\
1&0&0&0&0&0&0&-1&-1&-1&0&0&0&0&0&0\\
0&1&0&0&-1&-1&-1&0&0&0&0&0&0&0&0&0\\
0&0&0&0&0&0&0&-1&-1&-1&-1&-1&-1&-1&-1&-1\\
0&0&-1&0&0&0&0&0&1&0&0&0&0&0&0&0\\
0&0&0&0&-1&-1&-1&0&0&0&0&0&1&0&0&0\\
0&0&1&0&0&1&0&0&0&0&0&0&0&0&0&0\\
0&0&0&0&0&0&0&0&1&0&0&0&0&0&0&0\\
0&0&0&0&0&1&0&0&0&0&0&0&0&0&0&0\\
0&0&0&0&0&0&0&0&0&0&0&-1&-1&-1&-1&-1\\
0&0&0&-1&0&0&0&0&0&0&0&1&1&0&0&0\\
0&0&0&-1&0&0&0&0&0&0&0&0&0&0&0&0\\
0&0&0&1&0&0&0&0&0&0&0&0&0&0&0&0\\
0&0&0&0&0&0&0&1&0&0&0&0&0&0&0&0\\
0&0&0&0&1&0&0&0&0&0&0&0&0&0&0&0\\
0&0&0&0&0&0&0&0&0&1&0&0&0&0&0&0\\
0&0&0&0&0&0&1&0&0&0&0&0&0&0&0&0\\
0&0&0&0&0&0&0&0&0&0&0&0&0&1&0&0\\
0&0&0&0&0&0&0&0&0&0&0&0&0&0&0&1 \end{array} \right]
\]
We have rank$(M)=16$ and therefore nullity$(M^T)=19-16=3$. A matrix $B$ with with $\im B = \ker M^T$ and $\ker B = \{0\}$ is given by
\[
\tiny
B^T = \left[ 
\begin{array}{ccccccccccccccccccc} 
1 & 1 & 1 & 0 & 1 & 0 & 1 & 0 & 0 & 0 & 0 & 0 & 0 & 1 & 1 & 1 & 1 & 0 & 0 \\
0 & 0 & 0 & 0 & 0 & 0 & 0 & 0 & 0 & 0 & 0 & 1& 1 & 0 & 0 & 0 & 0 & 0 & 0 \\
1 & 1 & 1 & 0 & 0 & 0 & 0 & 1 & 1 & 0 & 0 & 0 & 0 & 1 & 1 & 1 & 1 & 0 & 0 
\end{array} 
\right] .
\]
By reducing $M^T$ to row echelon form, we obtain the following generalized inverse of $M^T$:
\[
\tiny
H = \left[ \begin{array}{cccccccccccccccc}
0&-1&0&0&0&-1&0&0&0&0&0&1&-1&0&0&0\\
1&-1&0&0&0&-1&0&0&0&0&0&1&-1&0&0&0\\
0&0&0&0&0&-1&0&0&0&0&0&1&-1&0&0&0\\
0&0&0&0&0&0&0&0&0&0&-1&0&0&0&0&0\\
0&0&-1&0&0&-1&1&0&0&0&0&0&0&0&0&0\\
0&0&0&0&0&0&0&0&0&0&0&-1&1&0&0&0\\
0&0&0&0&0&-1&1&0&0&0&0&0&0&0&0&0\\
1&-1&1&0&0&0&-1&0&0&1&-1&1&-1&0&0&0\\
0&0&0&0&0&0&0&0&0&0&0&0&0&0&0&0\\
0&0&0&0&0&0&0&0&0&0&1&0&0&0&-1&0\\
0&0&0&0&0&0&0&0&0&0&0&1&0&0&-1&0\\
0&0&0&-1&0&0&0&0&0&0&0&-1&0&0&1&0\\
0&0&0&0&0&0&0&0&0&0&0&0&0&0&0&0\\
1&-1&0&0&0&-1&0&1&0&0&-1&1&-1&0&0&0\\
0&0&0&0&1&-1&0&0&0&0&0&0&0&0&0&0\\
1&-1&0&0&0&-1&0&0&1&0&-1&1&-1&0&0&0\\
0&0&0&0&0&0&0&0&0&0&0&0&0&0&0&0\\
0&0&0&0&0&0&0&0&0&0&0&0&0&1&-1&0\\
0&0&0&0&0&0&0&0&0&0&0&0&0&0&-1&1
\end{array} \right]
\]
That is, $M^T H M^T = M^T$. 
From the graph~\eqref{mess1},
we obtain the tree constants $K=K(k^*,\fp)$:
\[\tiny \begin{array}{|l|l|l|}
\hline K_{1} = k_2 k_4 & K_{8} = (k_{10}+k_{11})k_{12}k_{14} & K_{15} = k_{15}k_{17}k_{18} \\
K_{2} = k_1 k_4 & K_{9} = k_9k_{12}k_{14} & K_{16} = (k_{22}+k_{23})k_{24}k_{30}((k_{28}+\fp_2+k_{31})k_{26}+k_{25}(k_{28}+k_{31}))\fp_1\\
K_{3} = k_1 k_3 & K_{10} = k_9k_{11}(k_{13}+k_{14}) & K_{17} = k_{21}(k_{22}+k_{23})k_{24}k_{30}((k_{28}+\fp_2+k_{31})k_{26}+k_{25}(k_{28}+k_{31}))\\
K_{4} = k_6 & K_{11} = k_9k_{11}k_{12} & K_{18} = k_{21}(k_{22}+k_{23})k_{24}(k_{25}+k_{26})k_{27}k_{30}\\
K_{5} = k_5 & K_{12} = (k_{16}+k_{17})k_{18}k_{20} & K_{19} = k_{21}(k_{22}+k_{23})(k_{25}+k_{26})k_{27}k_{30}\fp_2\\
K_{6} = k_8 & K_{13} = k_{15}k_{18}k_{20} & K_{20} = k_{21}k_{24}k_{30}((k_{28}+\fp_2+k_{31})k_{26}+k_{25}(k_{28}+k_{31}))\fp_1\\
K_{7} = k_7 & K_{14} = k_{15}k_{17}(k_{19}+k_{20})& K_{22} = k_{21}(k_{22}+k_{23})k_{24}k_{27}k_{30}\fp_2 \\
\hline \multicolumn{3}{|l|}{\begin{array}{l} K_{21} = ((((k_{28}+\fp_2+k_{31})k_{29}+(\fp_1+k_{27})k_{31}+\fp_2k_{27}+\fp_1(\fp_2+k_{28}))k_{26}+k_{25}((k_{28}+k_{31})k_{29}+(\fp_1+k_{27})k_{31}+\fp_1k_{28}))k_{23}\\ \hspace{1.5in} + k_{22}(((k_{28}+\fp_2+k_{31})k_{29}+k_{27}(k_{31}+\fp_2))k_{26}+k_{25}((k_{28}+k_{31})k_{29}+k_{31}k_{27})))k_{24}k_{21}\end{array}} \\
\hline
\end{array}\]

As a result, the parametrization~\eqref{eq:param} amounts to
\small
\begin{multicols}{3} 
\noindent 
$x_1 = \left( \frac{K_{3} K_{10} K_{19}}{K_{1} K_{8} K_{18}} \right) \fpp_1 \fpp_3\\ 
x_2 = \left( \frac{K_{3} K_{10} K_{19}}{K_{2} K_{8} K_{18} } \right) \fpp_1 \fpp_3\\
x_3 = \left( \frac{K_{10} K_{19}}{ K_{8} K_{18} } \right)\fpp_1 \fpp_3\\
x_4 = \left( \frac{K_{17}}{K_{16}} \right)\\
x_5 = \left( \frac{K_{5} K_{10}}{ K_{4} K_{11} }\right) \fpp_1\\
x_6 = \left( \frac{K_{18}}{K_{19}} \right)\\
x_7 = \left( \frac{K_{10}}{K_{11}} \right) \fpp_1\\
x_8 = \left( \frac{K_{3} K_{4} K_{11} K_{12} K_{17} K_{19} }{ K_{2} K_{5} K_{8} K_{15} K_{18} K_{16} } \right) \fpp_3\\
x_9 = \fpp_3\\
x_{10} = \left( \frac{K_{21}}{K_{17}} \right)\\
x_{11} = \left( \frac{K_{21}}{K_{18}} \right)\\
x_{12} = \left( \frac{K_{7} K_{18}}{ K_{6} K_{21} }\right) \fpp_2\\
x_{13} = \fpp_2\\
x_{14} = \left( \frac{K_{3} K_{10} K_{12} K_{17} K_{19}}{ K_{2} K_{8} K_{13} K_{18} K_{16} } \right) \fpp_1 \fpp_3\\
x_{15} = \left( \frac{K_{10}}{K_{9}} \right) \fpp_1 \fpp_3\\
x_{16} = \left( \frac{K_{3} K_{10} K_{12} K_{17} K_{19}}{ K_{2} K_{8} K_{14} K_{18} K_{16} } \right) \fpp_1 \fpp_3\\
x_{17} = \fpp_1 \fpp_3\\
x_{18} = \left( \frac{K_{21}}{K_{20}} \right) \\
x_{19} = \left( \frac{K_{21}}{K_{22}} \right)$ \end{multicols}
\noindent \normalsize 
with $K=K(k^*,\fp)$ as above
and $\fp_1, \fp_2, \fpp_1, \fpp_2, \fpp_3 > 0$. 
\end{exa}

\section{Outlook}
\label{sec:conclusions}

We have presented sufficient conditions for determining whether the set of positive equilibria of a generalized mass-action system coincides with the parametrized set of complex-balanced equilibria. We have also presented sufficient conditions for guaranteeing a positive parametrization of the set of complex-balanced equilibria and for effectively constructing the parametrization. Through an extension of network translations \cite{J2014}, we have shown how the result can be immediately applied to biochemical reaction networks, including the EnvZ-OmpR signaling pathway \cite{Sh-F} and shuttled WNT signaling pathway \cite{G-H-R-S}.

A number of potential avenues for further research naturally emerge from this work.
\begin{enumerate}
\item
Recent work on generalized mass-action systems has established sign conditions sufficient for the uniqueness of equilibrium points in compatibility classes \cite{BanajiPantea16, M-F-R-C-S-D}. In particular, when the steady state set is \emph{toric} or \emph{complex-balanced}, uniqueness and multistationarity may be established \cite{MR2012,M-D-S-C}. It is currently unclear, however, whether the extension to rational parametrizations in Theorem~\ref{main2} might be utilized to guarantee either uniqueness or multistationarity. 
\item
For GCRNs with nonzero kinetic deficiency ($\tilde \delta > 0$), statement 2 in Theorem~\ref{main2} guarantees that, if the parameters $\fp \in \RR_{> 0}^{E_0}$ can be chosen to satisfy the $\tilde \delta > 0$ conditions for complex-balancing, then the system has a monomial parametrization. It is currently unclear which conditions guarantee that a set of free parameters $\fp \in \RR_{> 0}^{E_0}$ may satisfy the $\tilde \delta > 0$ algebraic conditions on the rate parameters required for complex balancing.
\item
Even for biochemical networks of moderate size, it is difficult to determine a translation scheme for constructing a GCRN corresponding to the original CRN. Computational approaches to network translation have been investigated in \cite{J2015} and \cite{Tonello2017}. These works, however, rely on the definitions of a GCRN in \cite{MR2012} and of network translation in \cite{J2014}. Using the more general definitions in \cite{MR2014} would allow to extend the applicability of the computational approaches to a significantly broader class of networks.
\end{enumerate}

\subsection*{Acknowledgments}

This project began during the 
`SQuaRE' \emph{Dynamical properties of deterministic and stochastic models of reaction networks}
at the American Institute of Mathematics (AIM), San Jose, California,
in October 2017. 

SM was supported by the Austrian Science Fund (FWF), project P28406.
CP was supported by NSF-DMS award 1517577.

\end{document}